\documentclass[11pt]{article} 
\usepackage{amsfonts, amssymb, graphicx, amsmath}
\usepackage{fullpage}
\usepackage[usenames]{xcolor}

\usepackage[pdftex]{hyperref}
\hypersetup{
    colorlinks=true,       
    linkcolor=blue,        
    citecolor=red,        
    filecolor=blue,        
    urlcolor=blue          
} 
\usepackage[utf8]{inputenc} 
\usepackage{natbib}
\usepackage{graphicx} 
\usepackage[titletoc]{appendix}

\usepackage{booktabs} 
\usepackage{array} 
\usepackage{paralist} 
\usepackage{verbatim} 
\usepackage{subfig} 
\usepackage{fourier}

\newtheorem{theorem}{Theorem}
\newtheorem{corollary}{Corollary}

\newtheorem{lemma}{Lemma}
\newtheorem{example}{Example}
\newtheorem{remark}{Remark}
\newtheorem{definition}{Definition}
\newenvironment{proof}[1][Proof]{\textbf{#1.} }{\ \rule{0.5em}{0.5em}}

\newcommand{\RR}{\mathbb{R}}
\newcommand{\argmin}{\mathop{\rm arg\min}}

\title{Geometric Inference for General High-Dimensional Linear Inverse Problems}
\author{T. Tony Cai\footnote{The research of Tony Cai was supported in part by NSF Grants DMS-1208982 and DMS-1403708, and NIH Grant R01 CA127334.}, \;\,
Tengyuan Liang\footnote{Tengyuan Liang acknowledges the support of Winkelman Fellowship.} \; and \; Alexander Rakhlin\footnote{Alexander Rakhlin gratefully acknowledges the support of NSF
under grant CAREER DMS-0954737.} \\ \\
Department of Statistics\\ The Wharton School \\ University of Pennsylvania}
\date{}

\begin{document}
\maketitle

\begin{abstract}
This paper presents a unified geometric framework for the statistical analysis of a general ill-posed linear inverse model which includes  as special cases  noisy compressed sensing, sign vector recovery, trace regression, orthogonal matrix estimation, and noisy matrix completion. We propose computationally feasible convex programs for statistical inference including estimation, confidence intervals and hypothesis testing. A theoretical framework is developed to characterize the local estimation rate of convergence and to provide statistical inference guarantees. Our results are built based on the local conic geometry and duality. 
The difficulty of statistical inference is captured by the geometric characterization of the local tangent cone through 
the Gaussian width and Sudakov minoration estimate.

\end{abstract}

\section{Introduction}
\label{Sec.Intro}

Driven by a wide range of applications, high-dimensional linear inverse problems such as noisy compressed sensing,  sign vector recovery, trace regression, orthogonal matrix estimation, and noisy matrix completion have drawn significant recent interest in several fields, including statistics, applied mathematics, computer science, and electrical engineering.  These problems are often studied in a case-by-case fashion and the focus so far is mainly on estimation. Although similarities in the technical analyses have been suggested heuristically, a general unified theory for statistical inference including estimation, confidence intervals and hypothesis testing is still yet to be developed.

In this paper, we consider a general linear inverse model 
\begin{align}
\label{GLI.Model}
 Y = \mathcal{X}(M) + Z 
\end{align}
where $M \in \mathbb{R}^{p}$ is the vectorized version of the parameter of interest, $\mathcal{X}: \mathbb{R}^{p} \rightarrow \mathbb{R}^n$ is a linear operator, and $Z\in \mathbb{R}^n$ is a noise vector. We observe $(\mathcal{X}, Y)$ and wish to recover the unknown parameter $M$. A particular focus is on the high-dimensional setting where the ambient dimension $p$ of the parameter $M$ is much larger than the sample size $n$, i.e., the dimension of $Y$. In such a setting,  the parameter of interest $M$ is commonly assumed to have, with respect to a given atom set $\mathcal{A}$, a certain low complexity structure which captures the true dimension of the statistical estimation problem. A number of high-dimensional inference problems actively studied in the recent literature can be seen as special cases of this general linear inverse model.

\noindent
{\bf High Dimension Linear Regression/Noisy Compressed Sensing.} In high-dimensional linear regression, one observes $(X, Y)$ with 
\begin{align}
\label{sparse.model}
Y = X M + Z, 
\end{align}
where $Y\in \mathbb{R}^n$,  $X\in \mathbb{R}^{n\times p}$ with $p\gg n$, $M\in \mathbb{R}^p$ is a sparse signal, and $Z\in\mathbb{R}^n$ is a noise vector. The goal is to recover the unknown sparse signal  of interest $M\in\mathbb{R}^p$ based on the observation $(X, Y)$ through an efficient algorithm. Many estimation methods including $\ell_1$-regularized procedures such as the Lasso and Dantzig Selector have been developed and analyzed. See, for example, \cite{tibshirani1996regression,candes2007dantzig,bickel2009simultaneous,buhlmann2011statistics} and the references therein. Confidence intervals and hypothesis testing for high-dimensional linear regression have also been actively studied in the last few years. A common approach is to first construct a de-biased Lasso or de-biased scaled-Lasso estimator and then make inference based on the asymptotic normality of  low-dimensional functionals of the de-biased estimator. See, for example, \cite{buhlmann2013statistical, zhang2014confidence, van2014asymptotically, javanmard2014confidence}.
 
\noindent
{\bf Trace Regression.}  Accurate recovery of  a  low-rank matrix  based on a small number of linear measurements has a wide range of applications and has drawn much recent attention in several fields.
See, for example, \cite{recht2010guaranteed,koltchinskii2011neumann,rohde2011estimation,koltchinskii2011nuclear,candes2011tight}. In trace regression, one observes $(X_i, Y_i)$, $i=1, ..., n$  with
\begin{align}
\label{trace.reg}
Y_i = {\sf Tr}(X_i^T M) + Z_i,
\end{align}
where  $Y_i\in \RR$, $X_i\in \RR^{p_1\times p_2}$ are measurement matrices,  and $Z_i$ are noise. The goal is to recover  the unknown matrix $M \in \mathbb{R}^{p_1 \times p_2}$ which is assumed to be of low rank. Here the dimension of the parameter $M$ is $p\equiv p_1p_2 \gg n$. A number of constrained and penalized nuclear minimization methods have been introduced and studied in both the noiseless and noisy settings. See the aforementioned references for further details.


\noindent
{\bf Sign Vector Recovery.} The setting of  sign vector recovery is similar to the one for the high-dimensional regression except the signal of interest is a sign vector. More specifically, in sign vector recovery, one observes  $(X,Y)$ with
\begin{align}
\label{sign.reg}
Y = X M + Z
\end{align} 
where $Y \in \mathbb{R}^n, X \in \mathbb{R}^{n \times p}$, $M \in \{ +1,-1\}^p$ is a sign vector, and $Z \in \mathbb{R}^n$ is a noise vector. The goal is to recover the unknown sign signal of interest $M$. Exhaustive search over the parameter set is computationally prohibitive. The noiseless case of \eqref{sign.reg},  known as the generalized multi-knapsack problem \citep{khuri1994zero,mangasarian2011probability},  can be solved through an integer program which is known to be computationally difficult even for checking the uniqueness of the solution, see \citep{prokopyev2005complexity,valiant1986np}.

\noindent
{\bf Orthogonal Matrix Recovery.} In some applications the matrix of interest in trace regression is known to be an orthogonal/rotation matrix \citep{ten1977orthogonal,gower2004procrustes}. 
More specifically, in orthogonal matrix recovery, we observe $(X_i,Y_i)$, $i = 1,\ldots,n$ as in the trace regression model \eqref{trace.reg} 
where  $X_i\in \RR^{m \times m}$ are measurement matrices and  $M \in \mathbb{R}^{m \times m}$ is an orthogonal matrix. The goal is to recover the unknown $M$ using an efficient algorithm.  
Computational difficulties come in because of the non-convex constraint. See \cite{chandrasekaran2012convex}.

\noindent
{\bf Matrix Completion.} 
Matrix completion aims to recover a low-rank matrix based on observations of a subset of entries. It can be viewed as a special case of the trace regression model \eqref{trace.reg} with the measurement matrices of the form $e_{i_k}e_{j_k}^\intercal$ for $k=1, ..., n$, where $e_i$ is the $i$th standard basis vector, and $i_1,\cdots, i_n$ and $j_1, \cdots, j_n$ are  randomly  drawn with replacement from $\{1, \cdots, p_1\}$ and $\{1, \cdots, p_2\}$, respectively. That is, the individual entries of the matrix $M$ are observed at randomly selected positions. The goal is to recover the low-rank matrix $M$ based on the partial observations $Y$. See \cite{candes2009exact, recht2011simpler} for matrix recovery in the noiseless case and  \cite{candes2010matrix, chatterjee2012matrix, cai2013matrix} for the noisy case.



Other high-dimensional inference problems that are closely connected to the structured linear inverse model  \eqref{GLI.Model} include high-dimensional covariance matrix estimation where the covariance matrix of interest is banded/sparse/spiked  \citep{karoui2008operator,cai2010optimal,cai2013sparse,cai2014optimal},  sparse and low rank decomposition in robust principal component analysis \citep{candes2011robust}, and sparse noise and sparse parameter in demixing problem \citep{amelunxen2013living}, to name a few. We will discuss the connections in details in Section \ref{o.e}.

There are several fundamental questions for this general class of high-dimensional linear inverse problems.
\begin{itemize}
\item[]{\bf Statistical Questions:} How well can the parameter $M$ be estimated? What is the intrinsic difficulty of the estimation problem? How to provide inference guarantees for $M$, i.e., confidence intervals and hypothesis testing, in general?
\item[]{\bf Computational Questions:} Are there computationally efficient (polynomial time complexity) algorithms that are also sharp in terms of statistical estimation and inference? 
\end{itemize}

 \subsection{High-Dimensional Linear Inverse Problems}

Linear inverse problems have been well studied in the classical setting where the parameter of interest lies in a convex set. See, for example, \cite{tikhonov1977methods}, \cite{o1986statistical}, and \cite{johnstone1990speed}. In particular, for estimation of a linear functional over a convex parameter space, \cite{donoho1994statistical} developed an elegant geometric characterization of the minimax theory in terms of the modulus of continuity. 
However, the theory relies critically on the convexity assumption of the parameter space. As shown in \cite{cai2004adaptation,cai2004minimax}, the behavior of the functional estimation and confidence interval problems is significantly different even when the parameter space is the union of two convex sets. For the high-dimensional linear inverse problems considered in the present paper, the parameter space is highly non-convex and the theory and techniques developed in the classical setting are not readily applicable.

For high-dimensional linear inverse problems such as those mentioned earlier, the parameter space has low-complexity and exhaustive search often leads to the optimal solution in terms of statistical accuracy. However, it is computationally prohibitive and requires the prior knowledge of the true low complexity. In recent years, relaxing the problem to a convex program such as $\ell_1$ or nuclear norm minimization and then solving it with optimization techniques have proven to be a powerful approach in individual cases. 

Unified approaches to signal recovery recently appeared both in the applied mathematics literature \citep{chandrasekaran2012convex,amelunxen2013living,oymak2013simple} and in the statistics literature \citep{negahban2012unified}. \cite{oymak2013simple} studied the generalized LASSO problem through conic geometry with a simple bound in terms of the $\ell_2$ norm of the noise vector. \citep{chandrasekaran2012convex} introduced the notion of atomic norm to define a low complexity structure and showed that Gaussian width captures the minimum sample size required to ensure recovery.  \cite{amelunxen2013living} studied the phase transition for the convex algorithms for a wide range of problems. These papers suggested that the geometry of the local tangent cone determines the minimum number of samples to ensure successful recovery in the  noiseless or deterministic noise settings. \cite{negahban2012unified} studied the regularized-$M$ estimation with a decomposable norm penalty in the additive Gaussian noise setting.

Another line of research is focused on a detailed analysis of the Empirical Risk Minimization (ERM) \citep{lecue2013learning}. Here, the objective function is the excess risk for the squared error loss. The excess risk is shown to have the rate of $n^{-1/2}$ or $n^{-1}$, in terms of the sample size $n$. The analysis is based on the empirical processes indexed by the general subgaussian functional classes, with a proper localization radius around the best parameter. In addition to convexity, the ERM requires the prior knowledge on the size of the bounded parameter set of interest. This knowledge is not needed for the algorithm we propose in the present paper. 

Compared to estimation, there is a paucity of methods and theoretical results for confidence intervals and hypothesis testing for these linear inverse models. Specifically for high-dimensional linear regression, confidence intervals and significance testing have drawn increasing recent attention. \cite{buhlmann2013statistical} studied a bias correction method based on the ridge estimation, while \cite{zhang2014confidence} proposed bias correction via score vector using scaled Lasso as the initial estimator. \cite{van2014asymptotically,javanmard2014confidence} focused on de-sparsifying the Lasso via constructing a near ``inverse'' of the Gram matrix, one uses node-wise Lasso while the other uses an $\ell_\infty$ constrained quadratic programing, with similar theoretical guarantees. 
To the best of our knowledge, inference procedures for other high-dimensional linear inverse models are yet to be developed.

\subsection{Geometric Characterization of Linear Inverse Problems}

Under the linear inverse model \eqref{GLI.Model}, the parameter $M$ is assumed to have  certain low complexity structure with respect to a given atom set in a high-dimensional Euclidean space, which introduces a non-convex constraint. The non-convex constraint poses difficulty for the inverse problem. However, proper convex relaxation based on the general atom structure provides a computationally feasible solution. 
Our goal is to recover and make inference on the parameter $M$ based on the observation $(\mathcal{X}, Y)$ efficiently. This problem can also be framed in the language of geometric functional analysis \citep{ledoux1991probability,vershynin2011lectures}. For point estimation, we are interested in how the local convex geometry around the true parameter affects the estimation procedure and the intrinsic estimation difficulty, in terms of the local upper bound and the local minimax lower bound respectively. Note that local tangent cone plays a key role in our analysis. For statistical inference, we develop general procedures induced by the convex geometry, which answers inferential questions such as confidence intervals and hypothesis testing efficiently. We are also interested in the sample size condition induced by the local convex geometry for valid inference guarantees. 

Complexity measures such as Gaussian width and Rademacher complexity are well studied in the empirical processes theory \citep{ledoux1991probability,talagrand1996majorizing}, and are known to capture the difficulty of the estimation problem. Covering/Packing entropy and volume ratio \citep{yang1999information,vershynin2011lectures,ma2013volume} are also widely used in geometric functional analysis to measure the complexity. In this paper, we show how these geometric quantities affect the computationally efficient estimation/inference procedure, as well as the intrinsic difficulty of the estimation/inference problem.

Our main result can be summarized as follows. We propose unified convex algorithms for estimation and inference, and then analyze the theoretical properties for these algorithms. 
On the local tangent cone $T_{\mathcal{A}}(M)$ (the formal definition is given in \eqref{Tangent.Cone}, and $B_2^p$ below denotes Euclidean ball in $\mathbb{R}^p$), geometric quantities such as the Gaussian width $w(B_2^p \cap T_{\mathcal{A}}(M))$, Sudakov minoration estimate $e(B_2^p \cap T_{\mathcal{A}}(M))$, and volume ratio $v(B_2^p \cap T_{\mathcal{A}}(M))$ (defined in Section \ref{Geometric-Quantity.sec}) capture the rate of convergence of the linear inverse problem. In terms of the upper bound, with overwhelming probability, if $n \succsim w^2(B_2^p \cap T_{\mathcal{A}}(M))$, the estimation error under $\ell_2$ norm for our algorithm is of the rate 
\begin{align*}
\sigma  \frac{\gamma_{\mathcal{A}}(M) w(\mathcal{XA})}{\sqrt{n}} 
\end{align*}
where $\gamma_{\mathcal{A}}(M)$ is the local asphericity ratio defined in \eqref{Asphere.Ratio}.
The minimax lower bound for estimation under $\ell_2$ norm over the local tangent cone $T_{\mathcal{A}}(M)$satisfies
\begin{align*}
\sigma \left[ \frac{e(B_2^p \cap T_{\mathcal{A}}(M))}{\sqrt{n}}  \vee \frac{v(B_2^p \cap T_{\mathcal{A}}(M))}{\sqrt{n}} \right].
\end{align*}
For statistical inference, we establish valid asymptotic normality for any low-dimensional linear functional of the parameter $M$ under the condition
$$
\lim_{n,p \rightarrow \infty } \frac{\gamma^2_{\mathcal{A}}(M) w^2(\mathcal{XA})}{\sqrt{n}}  = 0,
$$
which can be compared to the condition for point estimation consistency
$$
\lim_{n,p \rightarrow \infty } \frac{\gamma_{\mathcal{A}}(M) w(\mathcal{XA})}{\sqrt{n}}  = 0.
$$
We remark on the critical difference on the sufficient conditions between valid inference and estimation consistency - more stringent condition on sample size $n$ is required for inference beyond estimation.
Intuitively, statistical inference is purely geometrized by Gaussian width and Sudakov minoration estimate.

\subsection{Our Contributions}

The main contributions of the present paper are two-fold.
\begin{itemize}
\item {\bf Unified convex algorithms for estimation and inference.}~ We propose a general computationally feasible convex program that provides near optimal rate of convergence simultaneously for a collection of high-dimensional linear inverse problems. We also provide a general convex feasibility program that leads to inference guarantees for any finite linear contrast, such as confidence intervals and hypothesis testing. 

\item {\bf Local geometric theory: Upper and lower bounds, confidence intervals and hypothesis testing.}~ A unified theoretical framework is provided for analyzing high-dimensional linear inverse problems based on the local conic geometry and duality. The point estimation and statistical inference are adaptive in the sense that the difficulty (rate of convergence, conditions on sample size, etc.) automatically adapts to the low complexity structure of the true parameter. Both the inference guarantee and estimation consistency are closely related and rely on conditions induced by the local conic geometry. 
It is shown that the minimax lower bound for estimation over the local tangent cone 
is captured by the Sudakov minoration estimate or volume ratio. 
The results geometrize statistical inference for general linear inverse problems with low complexity structure.
\end{itemize}

\subsection{Organization of the Paper}

The rest of the paper is structured as follows. In Section \ref{Sec.Prep}, after notation, definitions,  and basic convex geometry are reviewed, we formally present convex programs for recovering the parameter $M$, and for providing inference guarantees for $M$, based on the observation $(\mathcal{X}, Y)$. The properties of the proposed procedures are then studied in Section \ref{GE.Thy}. Under the Gaussian setting, a geometric theory is developed in terms of the local upper bound, the minimax lower bound as well as the confidence intervals and hypothesis testing. Applications to particular high-dimensional estimation problems are also included at the end of this section. Section \ref{Gen.Thy} extends the geometric theory beyond Gaussian. Relations between the upper and lower bounds are discussed. Further discussions appear in Section \ref{Sec.Dis}, and the proofs of the main results are given in Section \ref{Sec.Pf} and Appendix \ref{Sec.Tech.Lma} and \ref{Pf.Cor}.

\section{Preliminaries and Algorithms}
\label{Sec.Prep}

We review in this section notation and definitions that will be used in the rest of the paper. In particular, we introduce basics of convex geometry  including important geometric quantities that will be shown to be instrumental in characterizing the difficulty for statistical estimation and inference in later sections. We then collect some known results on the complexity measures,  Gaussian width, Sudakov estimate and volume ratio, that will be used repeatedly later. Finally, we will formally introduce our general estimation and inference programs based on the convex geometry and duality.
 
In this paper, we use $\| \cdot \|_{\ell_q}$ to denote the $\ell_q$ norm of a vector and use $B_2^p$ to denote the unit Euclidean ball in $\RR^p$. For a matrix $M$,  denote by $\|M\|_F$, $\| M \|_*$, and  $\| M\|$ the Frobenius norm, nuclear norm, and spectral norm of  $M$ respectively. When there is no confusion, we also denote $\| M \|_F  = \| M \|_{\ell_2}$ for a matrix $M$. 
For a vector $V \in \mathbb{R}^p$, denote its transpose by $V^*$. The inner product  on vectors is defined as usual $\langle V_1, V_2 \rangle = V_1^* V_2$. For matrices $\langle M_1, M_2 \rangle = {\sf Tr}(M_1^* M_2) = {\sf Vec}(M_1)^* {\sf Vec}(M_2)$, where ${\sf Vec}(M) \in \mathbb{R}^{pq}$ denotes the vectorized version of matrix $M \in \mathbb{R}^{p \times q}$. $\mathcal{X}: \mathbb{R}^p \rightarrow \mathbb{R}^n$ denotes a linear operator from $\mathbb{R}^p$ to $\mathbb{R}^n$. Following the notation above, $M^* \in \mathbb{R}^{q \times p}$ is the adjoint (transpose) matrix of $M$ and $\mathcal{X}^*: \mathbb{R}^n \rightarrow \mathbb{R}^p$ is the adjoint operator of $\mathcal{X}$ such that $\langle \mathcal{X} (V_1),V_2 \rangle = \langle V_1, \mathcal{X}^* (V_2) \rangle $.


For a convex compact set $K$ in a metric space with the metric $d$, we say that $S \subset K$ is an $\epsilon$-covering set if $\forall x \in K$, $\exists y \in S$ such that $d(x,y) < \epsilon$. And we say that $S \subset K$ is an $\epsilon$-packing set if $\forall x,y \in S, x\neq y$, $d(x,y) \geq \epsilon$. The $\epsilon$-entropy for a convex compact set $K$ with respect to the metric $d$ is denoted in the following way: $\epsilon$-packing entropy $\log \mathcal{M}(K,\epsilon,d)$ is the logarithm cardinality of the largest $\epsilon$-packing set, and $\epsilon$-covering entropy $\log \mathcal{N}(K,\epsilon,d)$ is the logarithm cardinality of the smallest $\epsilon$-covering set with respect to metric $d$. A well known result is $\mathcal{M}(K,2\epsilon,d) \leq \mathcal{N}(K,\epsilon,d) \leq \mathcal{M}(K,\epsilon,d)$. When the metric $d$ is the usual Euclidean distance, we will omit $d$ in $\mathcal{M}(K,\epsilon,d)$ and $\mathcal{N}(K,\epsilon,d)$ and simply write $\mathcal{M}(K,\epsilon)$ and $\mathcal{N}(K,\epsilon)$. 

For two sequences of positive numbers $\{a_n\}$ and $\{b_n\}$, we denote $a_n \gtrsim b_n$ if there exists a constant $c_0$ such that $\frac{a_n}{b_n} \geq c_0$ for all $n$ and  $a_n \lesssim b_n$ if there exists a constant $C_0$ such that $\frac{a_n}{b_n} \leq C_0$ for all $n$. We write $a_n \asymp b_n$ 
if $a_n \gtrsim b_n$ and $a_n \lesssim b_n$. Throughout the paper, $c, C, c_0, C_0$ denote constants that may vary from place to place.

\subsection{Basic Convex Geometry}
\label{Basic.Geo}
We consider the linear inverse model \eqref{GLI.Model} in the high-dimensional setting where the dimension $p$ can possibly be much larger than the sample size $n$ and the parameter of interest $M$  lies in a certain ``low complexity'' space. Examples include  sparsity in noisy compressed sensing and low rank in trace regression and matrix completion. The linear operator $\mathcal{X}$ in the model \eqref{GLI.Model} can be viewed as a matrix $\mathcal{X} \in \mathbb{R}^{n\times p}$. Without loss of generality, we assume $\mathcal{X}$ is standardized to have unit column $\ell_2$ norm. The noise vector $Z\in \mathbb{R}^n$ is assumed to have the noise level $\sigma/\sqrt{n}$ and the covariance  matrix $\frac{\sigma^2}{n} {\bf I}_{n}$.

The notion of low complexity is based on a collection of basic atoms. We denote the collection of these basic atoms as an atom set $\mathcal{A}$, either countable or uncountable, as illustrated in Figure \ref{fig:atom_set}. A parameter $M$ is of complexity $k$ in terms of the atoms in $\mathcal{A}$ if $M$ can be expressed as a linear combination of at most $k$ atoms in $\mathcal{A}$, i.e., there exists a decomposition 
\begin{align*}
M = \sum_{a\in \mathcal{A}} c_a(M) \cdot a, ~\text{where}~ \sum_{a \in \mathcal{A}} 1_{\{c_a(M) \neq 0\}} \leq k 
\end{align*}


\begin{figure}[pht]
\centering
\begin{minipage}{.45\textwidth}
  \centering
  \includegraphics[width=\textwidth]{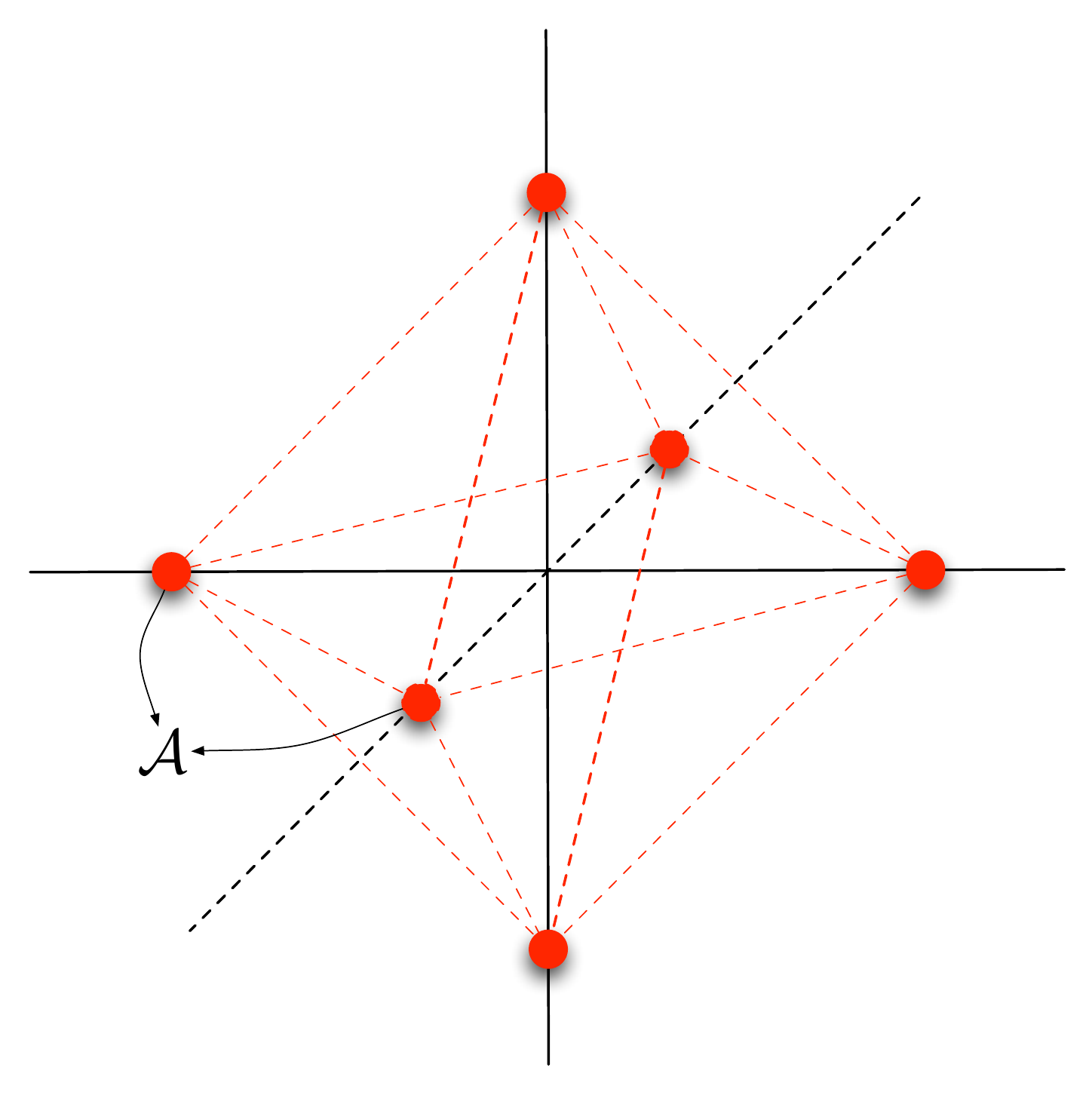}
  \caption{\small Atom set illustration. The red dots denote atoms. This particular example illustrates the atoms being basis vectors for sparse regression.}
  \label{fig:atom_set}
\end{minipage}%
\hspace{0.4cm}
\begin{minipage}{.45\textwidth}
  \centering
   \includegraphics[width=0.95\textwidth]{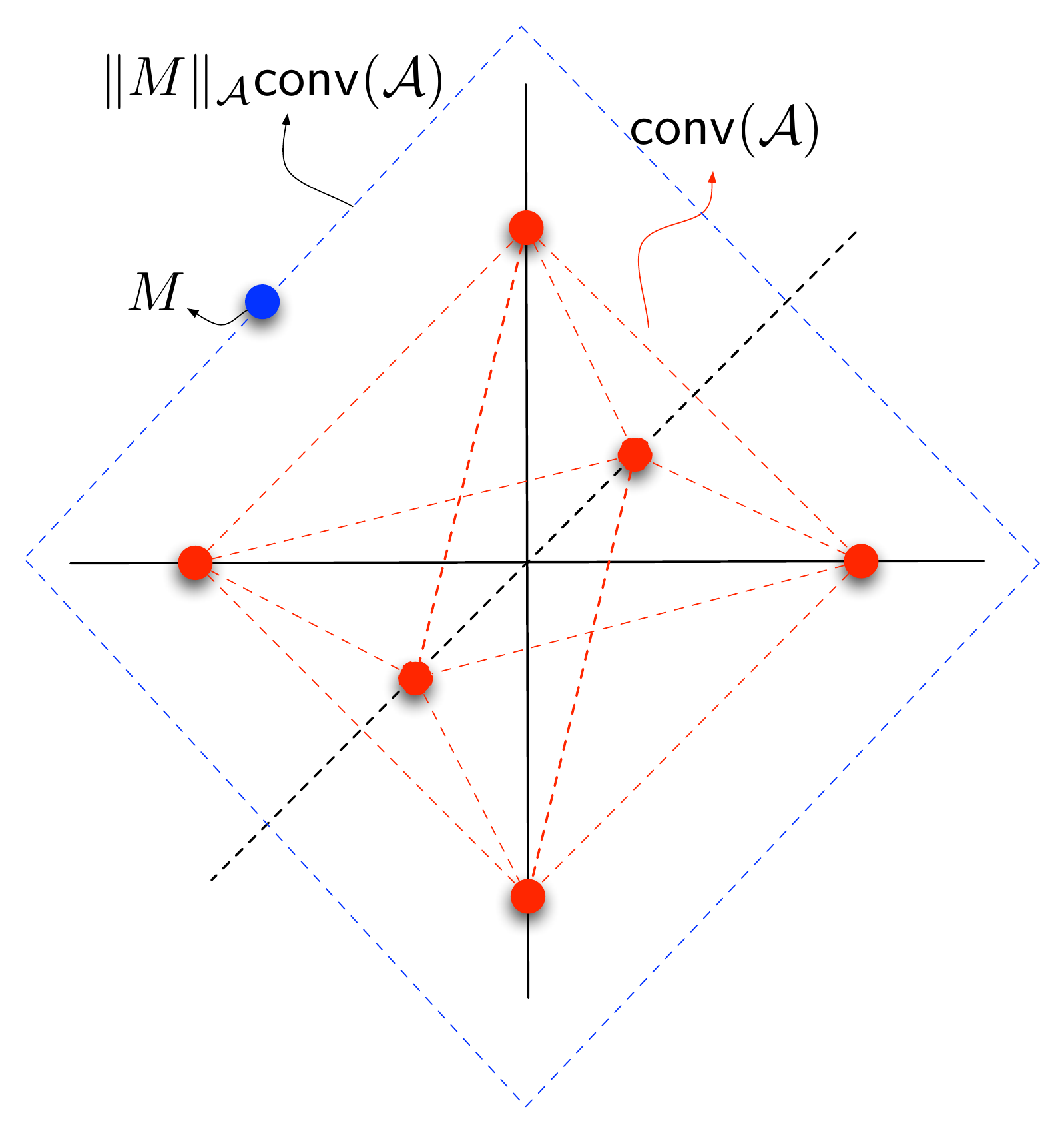}
    \caption{\small Atomic norm illustration. The red dashed line denotes the convex hull of atoms set. The blue dashed line denotes the scaled convex hull where $M$ lies in.}
    \label{fig:convex_hull}
\end{minipage}
\end{figure}

In convex geometry \citep{pisier1999volume}, the Minkowski functional (gauge) of a symmetric convex body $K$ is defined as 
\begin{align*}
\| x \|_K = \inf\{ t>0: x \in t K\}.
\end{align*}
Let $\mathcal{A}$ be a collection of atoms that is a compact subset of $\mathbb{R}^p$. We assume that  the elements of $\mathcal{A}$ are extreme points of the convex hull ${\sf conv}(\mathcal{A})$ (in the sense that for any $x \in \mathbb{R}^p$, $\sup \{ \langle x, a \rangle: a\in \mathcal{A}  \}  = \sup\{ \langle x,a \rangle : a \in {\sf conv}(\mathcal{A}) \}$). The atomic norm $\| x \|_{\mathcal{A}}$ for any $x \in \mathbb{R}^p$ is defined as the gauge of ${\sf conv}(\mathcal{A})$ (see Figure \ref{fig:convex_hull}):
\begin{align*}
\| x \|_{\mathcal{A}} = \inf \{ t> 0 : x \in t~ {\sf conv}(\mathcal{A})\}.
\end{align*}

As noted in \cite{chandrasekaran2012convex}, the atomic norm can also be written as
\begin{align}
\label{ATOM.Norm}
\| x \|_\mathcal{A} = \inf \left\{ \sum_{a \in \mathcal{A}} c_a : x = \sum_{a \in \mathcal{A}} c_a \cdot a, ~c_a \geq 0  \right\}.
\end{align}
The dual norm of this atomic norm is defined in the following way (since  the atoms in $\mathcal{A}$ are the extreme points of ${\sf conv}(\mathcal{A})$),
\begin{align}
\label{ATOM.Dual}
\| x \|_{\mathcal{A}}^* & = \sup \{ \langle x, a \rangle: a\in \mathcal{A}  \}  = \sup\{ \langle x,a \rangle : \| a \|_\mathcal{A} \leq 1 \}.
\end{align}
We have the following (``Cauchy-Schwarz'') symmetric relation for the norm and its dual 
\begin{align}
\label{Cauchy.Schwarz}
\langle x,y \rangle \leq \|x\|_{\mathcal{A}}^* \|y\|_{\mathcal{A}}.
\end{align}

\begin{figure}[pht]
\centering
\begin{minipage}{.45\textwidth}
  \centering
  \includegraphics[width=0.75\textwidth]{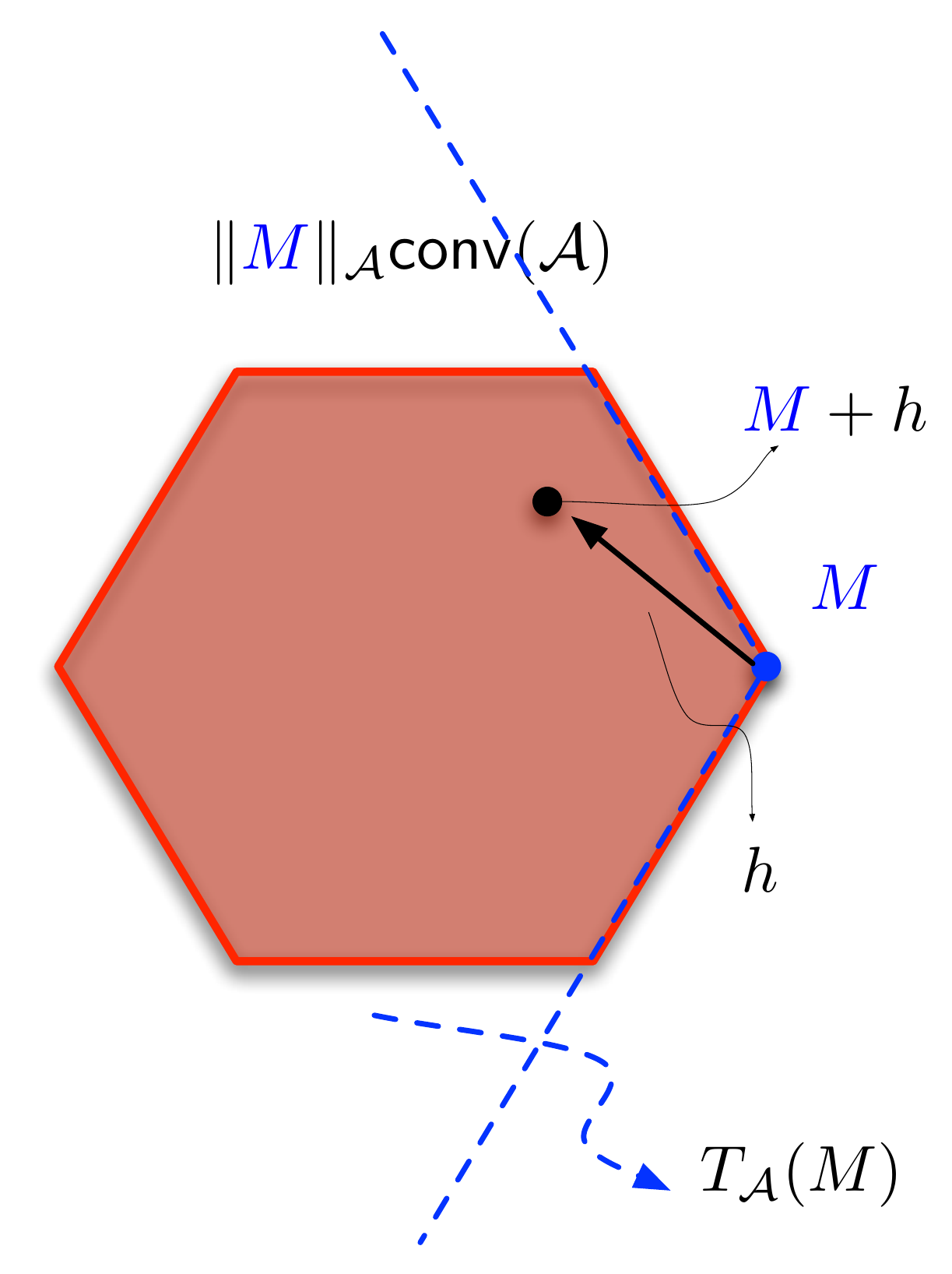}
  \caption{\small Tangent cone general illustration 2D. The red shaped area is the scaled convex hull of atom set. The blue dashed line forms the tangent cone at $M$. Black arrow denotes the possible directions inside the cone.}
  \label{fig:cone2d}
\end{minipage}%
\hspace{0.4cm}
\begin{minipage}{.45\textwidth}
  \centering
   \includegraphics[width=\textwidth]{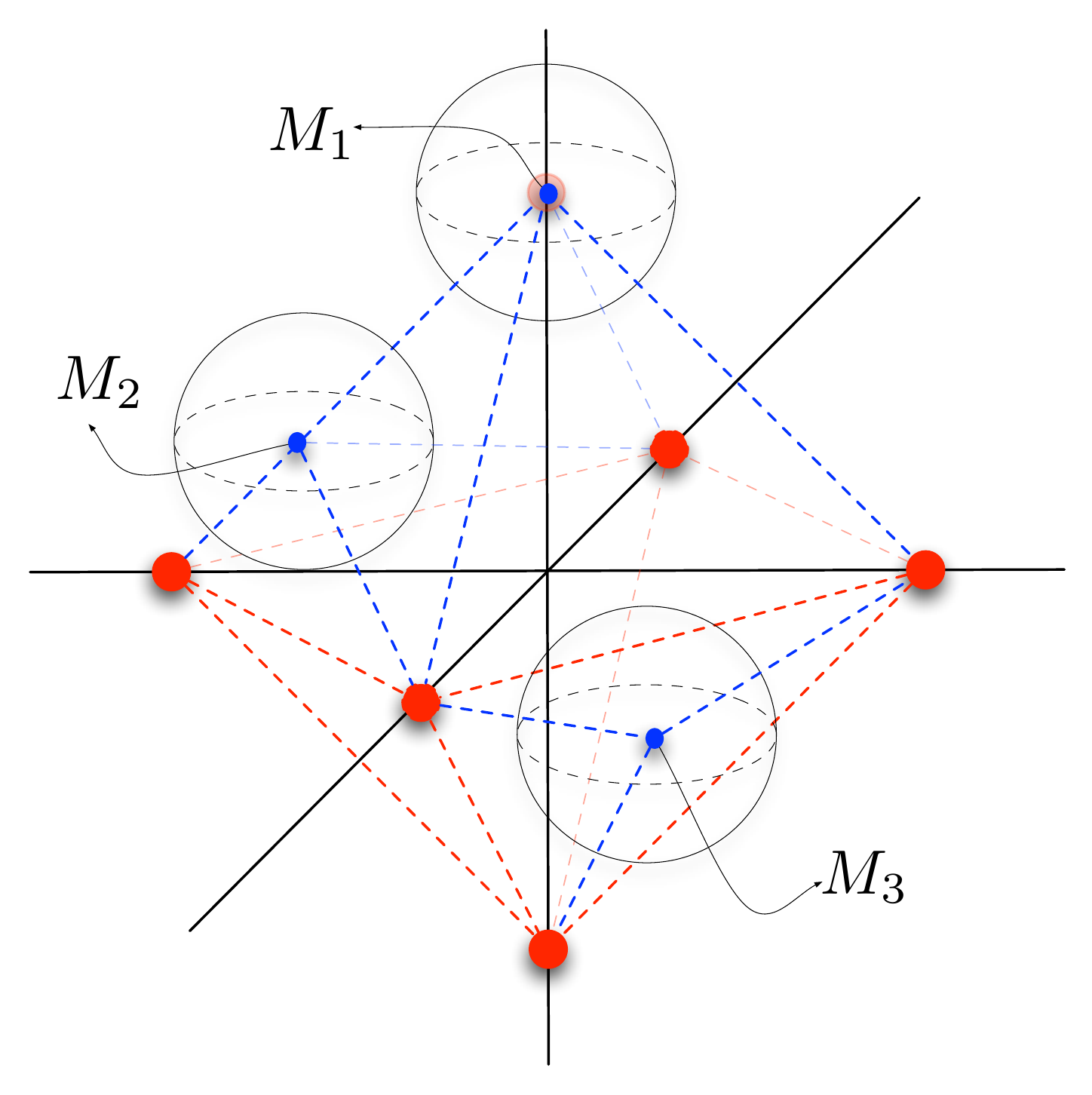}
  \caption{\small Tangent cone illustration 3D for sparse regression. For three possible locations $M_i, 1\leq i\leq3$, the tangent cone are different, with cones becoming more complex as $i$ increases. }
  \label{fig:cone3d}
\end{minipage}
\end{figure}

It is clear that the unit ball with respect to the atomic norm $\| \cdot \|_{\mathcal{A}}$ is the convex hull of the set of atoms $\mathcal{A}$. The {\bf tangent cone} at $x$ with respect to the scaled unit ball $\| x \|_{\mathcal{A}} ~ {\sf conv}(\mathcal{A})$ is defined to be (see Figures \ref{fig:cone2d} and \ref{fig:cone3d})
\begin{align}
\label{Tangent.Cone}
T_{\mathcal{A}} (x) = {\sf cone}\left\{h: \|x+h\|_{\mathcal{A}} \leq \| x\|_{\mathcal{A}}\right\}.
\end{align}
Also known as a recession cone, $T_{\mathcal{A}} (x)$ is the collection of directions where the atomic norm becomes smaller. This tangent cone $T_{\mathcal{A}} (x)$ determines the geometric property of the neighborhood around the true parameter $M$, and thus the complexity of this cone will affect the difficulty of the recovery problem. The cone is unbounded, but we can look at the cone intersected with the unit ball $B_2^p \cap T_{\mathcal{A}}(M)$ in analyzing the complexity of the cone. Figure \ref{fig:cone2d} provides an intuitive illustration where the red shaded area is the scaled atomic norm ball, $M$ is the true parameter, the black arrow denotes one vector inside the tangent cone, and the region enclosed by the blue dashed lines is the $T_{\mathcal{A}}(M)$.

In order to better illustrate the general model and notion of low complexity, it is helpful to look at the atom set, atomic norm and tangent cone geometry in a few examples.

\begin{example}{\rm
For sparse signal recovery in high-dimensional linear regression, the atom set consists of the unit basis vectors $\{ \pm e_i\}$, the atomic norm is the vector $\ell_1$ norm, and its dual norm is the vector $\ell_\infty$ norm. The convex hull ${\sf conv}(\mathcal{A})$ is called the cross-polytope. Figure \ref{fig:cone3d} illustrates this tangent cone for 3D $\ell_1$ norm ball for 3 different cases $T_{\mathcal{A}}(M_i),1\leq i\leq 3$. The  ``angle'' or ``complexity'' of the local tangent cone determines the difficulty of recovery. Most of the previous work showed that the algebraic characterization (sparsity) of the parameter space drives the global rate, and we are arguing that the geometric characterization through the local tangent cone provides an intuitive and refined local approach to high-dimensional linear inverse problem.
}
\end{example}

\begin{example}{\rm
In trace regression and matrix completion, the goal is to recover low rank matrices. In such settings,  the atom set consists of the rank one matrices (matrix manifold) $\mathcal{A} = \{ u v^*: \|u\|_{\ell_2} = 1, ~\|v\|_{\ell_2}=1 \}$ and the atomic norm is the nuclear norm and the dual norm is the spectral norm. The convex hull ${\sf conv}(\mathcal{A})$ is called the nuclear norm ball of matrices. The position of the true parameter on the scaled nuclear norm ball determines the geometry of the local tangent cone, thus affecting the estimation difficulty. 
}
\end{example}

\begin{example}{\rm
In integer programming, one would like to recover the sign vectors whose entries take on values $\pm 1$. The atom set is all sign vectors (cardinality $2^p$) and the convex hull ${\sf conv}(\mathcal{A})$ is the hypercube. Tangent cones for each parameter have the same structure in this case.
}
\end{example}

\begin{example}{\rm
In orthogonal matrix recovery, the matrix of interest is constrained to be orthogonal.  In this case, the atom set is all orthogonal matrices and the convex hull ${\sf conv}(\mathcal{A})$ is the spectral norm ball. Similar to sign vector recovery, the local tangent cones for each orthogonal matrix share similar geometric property.
}
\end{example}



\subsection{Gaussian Width, Sudakov Estimate, and Other Geometric Quantities}
\label{Geometric-Quantity.sec}

Our theoretical analysis relies on several key geometric quantities. We first introduce two complexity measures, the Gaussian width and Sudakov estimate. 

\begin{definition}[Gaussian Width]
For a compact set $K \in \mathbb{R}^p$, the Gaussian width is defined as
\begin{align}
\label{Gaussian.Width}
w(K) := \mathbb{E}_{g} \left[ \sup_{v \in K} \langle {g}, v \rangle \right].
\end{align}
where $g \sim N(0,I_{p})$ is the standard multivariate Gaussian vector.
\end{definition}

Gaussian width quantifies the probability that a randomly oriented subspace misses a convex subset. It was introduced in Gordon's analysis \citep{gordon1988milman}, and was shown recently to play a crucial rule in linear inverse problems in various noiseless or deterministic noise settings, see, for example,  \cite{amelunxen2013living}. Explicit upper bounds on the Gaussian width for different convex sets have been given in \cite{chandrasekaran2012convex,amelunxen2013living}.  For example, if $M \in \mathbb{R}^p$ is a $s-$sparse vector, $w(B_2^p \cap T_{\mathcal{A}}(M)) \lesssim \sqrt{s \log p/s}$. When $M \in \mathbb{R}^{p \times q}$ is a rank-$r$ matrix, $w(B_2^p \cap T_{\mathcal{A}}(M)) \lesssim \sqrt{r(p+q-r)}$. For sign vector in $\mathbb{R}^p$, $w(B_2^p \cap T_{\mathcal{A}}(M)) \lesssim \sqrt{p}$, while for orthogonal matrix in $\mathbb{R}^{m \times m}$, $w(B_2^p \cap T_{\mathcal{A}}(M)) \lesssim \sqrt{m(m-1)}$. See  Section 3.4 propositions 3.10-3.14 in \cite{chandrasekaran2012convex} for detailed calculations.
The Gaussian width as a complexity measure of the local tangent cone will be used in the upper bound analysis in Sections \ref{GE.Thy} and \ref{Gen.Thy}.

\begin{definition}[Sudakov Minoration Estimate]
The Sudakov estimate of a compact set $K \in \mathbb{R}^p$ is defined as
\begin{align}
\label{Covering.Entropy}
e(K) & := \sup_{\epsilon}~ \epsilon \sqrt{\log \mathcal{N}(K, \epsilon)} .
\end{align} 
where $\mathcal{N}(K,\epsilon)$ denotes the $\epsilon-$covering number of set $K$ with respect to the Euclidean norm. 
\end{definition}
Sudakov estimate has been widely known in the literature to capture the complexity of a general functional class \citep{yang1999information}. Through balancing the cardinality of the covering set at scale $\epsilon$ and the covering radius $\epsilon$, Sudakov estimate defines the best radius $\epsilon$ that maximizes $$\epsilon \sqrt{\log \mathcal{N}(B_2^p \cap T_{\mathcal{A}}(M),\epsilon)},$$ thus determines the complexity of the set $T_{\mathcal{A}}(M),\epsilon)$. Sudakov estimate as a complexity measure of the local tangent cone is useful for the minimax lower bound analysis.


 
\begin{figure}[pht]
\centering
\begin{minipage}{.4\textwidth}
  \centering
  \includegraphics[width=0.8\textwidth]{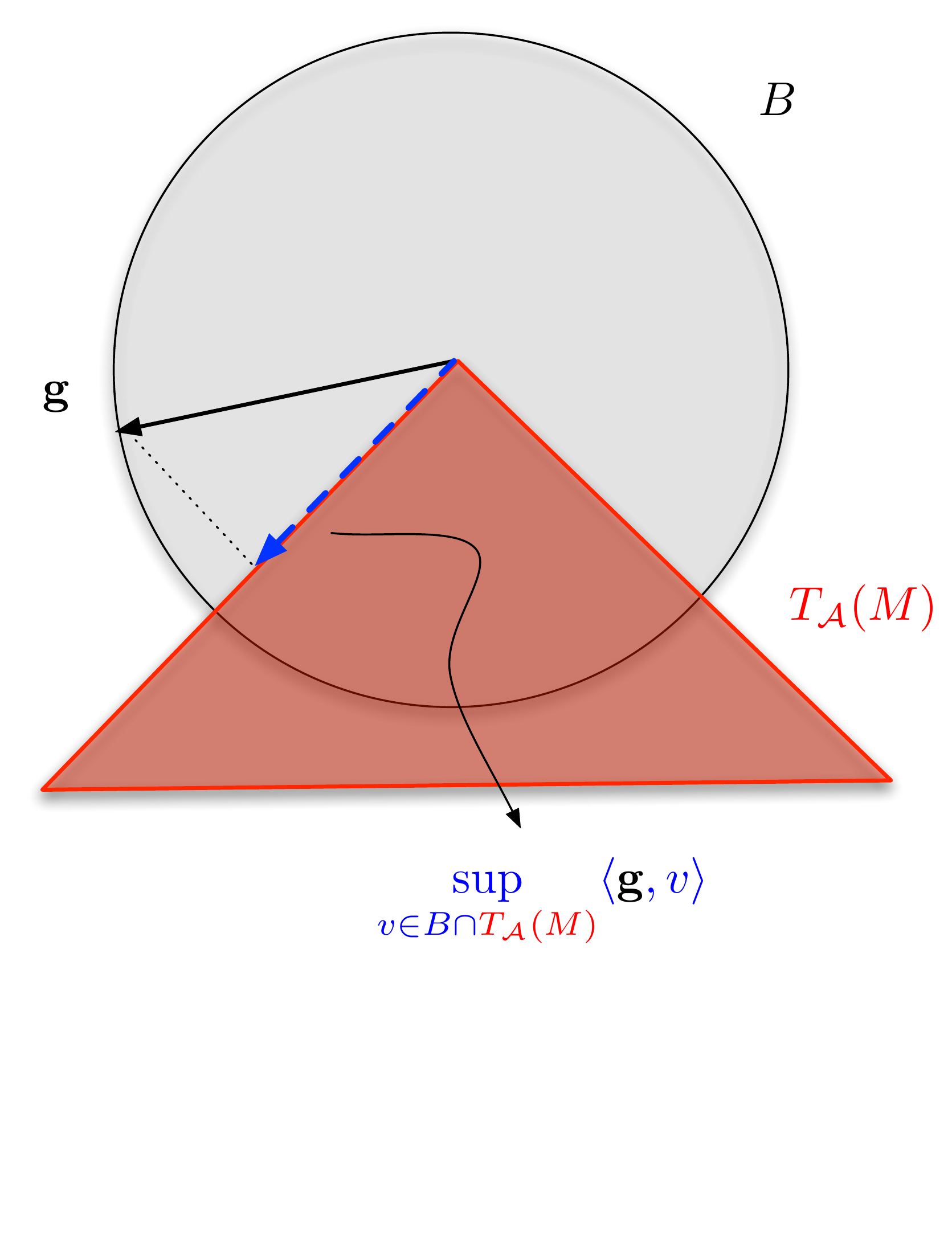}
  \vspace{-40pt}
  \caption{\small Gaussian width.}
  \label{fig:gaussian_width}
\end{minipage}
\hspace{0.3cm}
\begin{minipage}{.41\textwidth}
  \centering
   \includegraphics[width=0.8\textwidth]{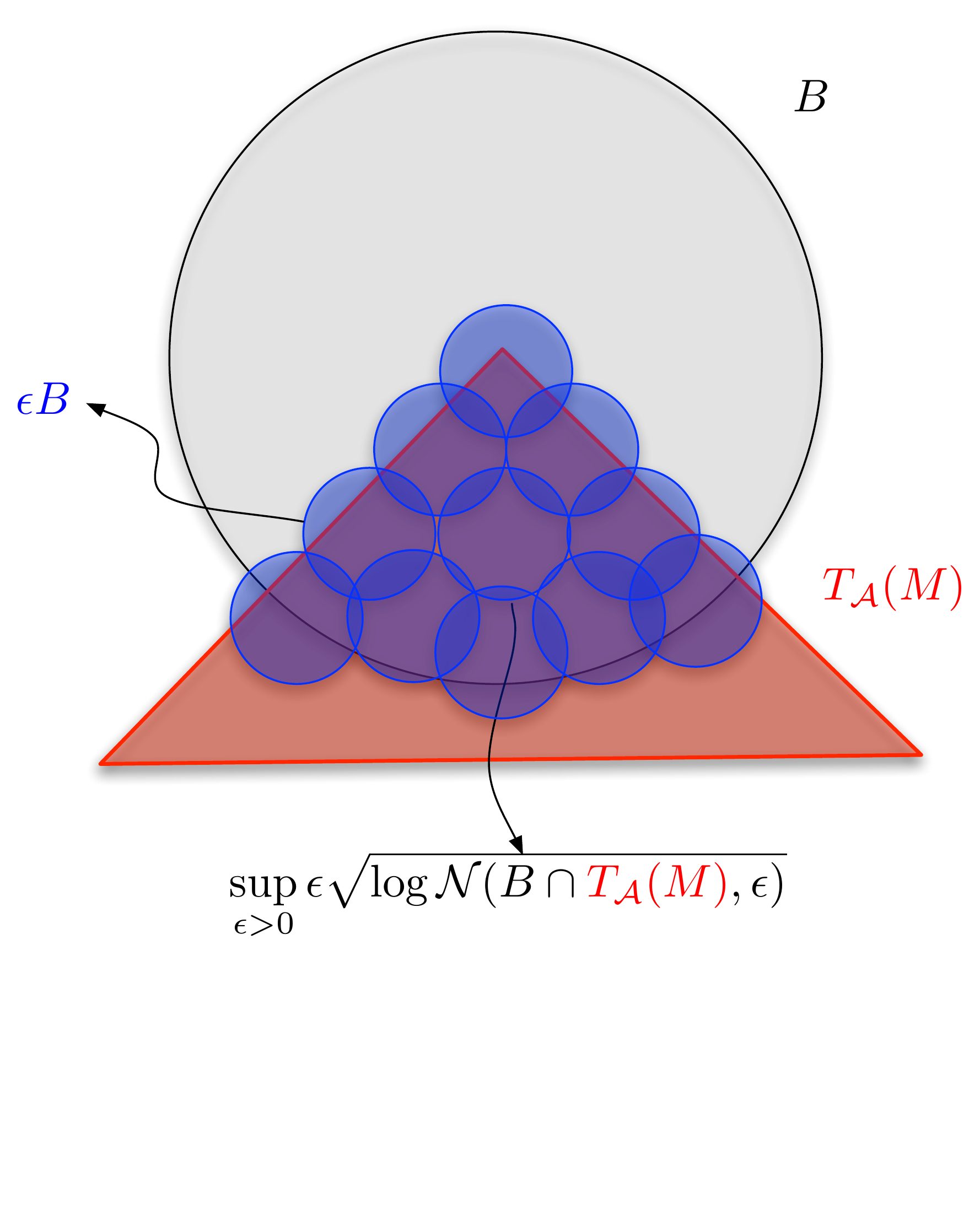}
     \vspace{-40pt}
  \caption{\small Sudakov estimate.}
  \label{fig:covering_width}
\end{minipage}
\end{figure}

The following Sudakov minoration and Dudley entropy integral \citep{dudley1967sizes,ledoux1991probability} show how the Gaussian width $w(\cdot)$ and Sudakov estimate $e(\cdot)$, both geometric quantities, are related to each other. 
\begin{lemma}[Sudakov Minoration and Dudley Entropy Integral]
\label{Sudakov.Lemma}
For any compact subset $K \subseteq \mathbb{R}^p$, there exist a universal constant $c>0$ such that
\begin{align}
c \cdot e(K) \leq w(K) \leq 24 \int_0^\infty  \sqrt{\log \mathcal{N}(K,\epsilon)} d\epsilon.
\end{align}
\end{lemma}

In the literature, another complexity measure, volume ratio has also been used to characterize the minimax lower bounds \citep{ma2013volume}. Volume ratio has been studied in \cite{pisier1999volume} and \cite{vershynin2011lectures}. For a convex set $K \in \mathbb{R}^p$, volume ratio used in the present paper is defined as follows.
\begin{definition}[Volume Ratio]
The volume ratio is defined as
\begin{align}
\label{Volume.Ratio}
v(K) &: = \sqrt{p} \left( \frac{{\sf vol}(K)}{ {\sf vol} (B_2^p)} \right)^{\frac{1}{p}}
\end{align}
\end{definition}
 

The following Urysohn's inequality, which is proved through Brunn-Minkowski Theorem, links the Gaussian width $w(\cdot)$ with the volume ratio $v(\cdot)$. 
\begin{lemma}[Urysohn's Inequality]
\label{Urysohn.Lemma}
Let $K$ be a compact subset of $\mathbb{R}^p$. Then
\begin{align*}
v(K) \leq w(K)
\end{align*}
with the equality achieved if and only if $K$ is the $\ell_2$ ball $B_2^p$.
\end{lemma}

The recovery difficulty of the linear inverse problem also depends on other geometric quantities defined on the local tangent cone $T_{\mathcal{A}}(M)$: the local isometry constants $\phi_{\mathcal{A}}(M,\mathcal{X})$ and $ \psi_{\mathcal{A}}(M,\mathcal{X})$ and the local asphericity ratio $\gamma_{\mathcal{A}}(M)$.
The {\bf local isometry constants} are defined for the local tangent cone at the true parameter $M$ as 
\begin{align}
\label{Iso.Const}
& \phi_{\mathcal{A}}(M,\mathcal{X}):= \inf\left\{ \frac{\|\mathcal{X} (h)\|_{\ell_2}}{ \| h \|_{\ell_2}} : h\in T_{\mathcal{A}}(M),h\neq 0 \right\}  \\
& \psi_{\mathcal{A}}(M,\mathcal{X}):= \sup\left\{ \frac{\|\mathcal{X} (h)\|_{\ell_2}}{ \| h \|_{\ell_2}} : h\in T_{\mathcal{A}}(M),h\neq 0 \right\}.
\end{align}
The local isometry constants measure how well the linear operator preserves the $\ell_2$ norm within the local tangent cone. Intuitively, the larger the $\psi$ or the smaller the $\phi$ is, the harder the recovery is.  
We will see later that the local isometry constants are determined by the Gaussian width under the Gaussian ensemble design. 


The {\bf local asphericity ratio} is defined as
\begin{align}
\label{Asphere.Ratio}
\gamma_{\mathcal{A}}(M):= \sup \left\{ \frac{\| h \|_{\mathcal{A}}}{\| h\|_{\ell_2}} :  h\in  T_{\mathcal{A}}(M),h\neq 0 \right\},
\end{align}
which measures how extreme the atomic norm is relative to the $\ell_2$ norm within the local tangent cone. 


\subsection{Point Estimation via Convex Relaxation}

We now return to the linear inverse model \eqref{GLI.Model} in the high-dimensional setting.
Suppose we observe $(\mathcal{X}, \; Y)$ as in \eqref{GLI.Model} where the parameter of interest $M$ is assumed to  have low complexity with respect to a given atom set $\mathcal{A}$. The low complexity of $M$ introduces a non-convex constraint, which leads to serious computational difficulties if solved directly. Convex relaxation is an effective and natural approach in such a setting. We propose a generic convex constrained minimization procedure induced by the atomic norm and the corresponding dual norm to estimate $M$:
\begin{align}
\label{CST.Min}
\hat{M} =  \argmin_{M} \left\{  \| M \|_{\mathcal{A}} : \; \| \mathcal{X}^*(Y - \mathcal{X} (M) ) \|_{\mathcal{A}}^* \leq \lambda \right\}
\end{align}
where $\lambda$ is a tuning parameter (localization radius) that depends on the sample size, noise level, and geometry of the atom set $\mathcal{A}$. An explicit formula for $\lambda$ is given in \eqref{Tun.Param} in the case of Gaussian noise. Intuitively, the atomic norm minimization \eqref{CST.Min} is a convex relaxation to the low complexity structure and $\lambda$ specifies the localization scale given the noise distribution. This generic convex program utilizes the duality and recovers the low complexity structure adaptively.  The Dantzig selector for high-dimensional sparse regression \citep{candes2007dantzig} and the constrained nuclear norm minimization \cite{candes2011tight} for trace regression are particular examples of \eqref{CST.Min}. The properties of the estimator $\hat M$ will be investigated in Sections \ref{GE.Thy} and \ref{Gen.Thy}.

\subsection{Statistical Inference via Feasibility of Convex Program}

In the high-dimensional setting, $p$-values as well as confidence intervals are important inferential questions beyond point estimation. In this section we will show how to perform statistical inference for the linear inverse model \eqref{GLI.Model}. Let $M \in \mathbb{R}^p$ be the vectorized parameter of interest, and $\{e_i, 1\leq i \leq p\}$ are the corresponding basis vectors. Consider the following convex feasibility problem for matrix $\Omega \in \mathbb{R}^{p \times p}$, where each row $\Omega_{i \cdot}$ satisfies
\begin{align}
	\label{FES.Con}
	\| \mathcal{X}^* \mathcal{X} \Omega_{i\cdot}^* - e_i \|_{\mathcal{A}}^* \leq \eta, ~~\forall 1\leq i\leq p
\end{align}
where $\eta$ is some tuning parameter that depends on the sample size and geometry of the atom set $\mathcal{A}$. One can also solve a stronger version of the above convex program for $\eta\in \mathbb{R}, \Omega \in \mathbb{R}^{p \times p}$ simultaneously 
\begin{align}
\label{FES2.Con}
	(\Omega, \eta_n) =  \argmin_{\Omega,\eta} \left\{ \eta: \;  \| \mathcal{X}^* \mathcal{X} \Omega_{i\cdot}^* - e_i \|_{\mathcal{A}}^* \leq \eta,  ~~\forall  1\leq i\leq p \right\}.
\end{align}


Built upon the constrained minimization estimator $\hat{M}$ in \eqref{CST.Min} and feasible matrix $\Omega$ in \eqref{FES2.Con}, the de-biased estimator for inference on parameter $M$ is defined as
\begin{align}
	\label{DB.Est}
	\tilde{M} := \hat{M} + \Omega \mathcal{X}^* (Y - \mathcal{X}(\hat{M})).
\end{align}
We will establish the asymptotic normality for finite linear contrast $\langle v, M \rangle$, where $v \in \mathbb{R}^p, \| v \|_{\ell_2} = 1, \| v \|_{\ell_0} \leq k$, $k$ does not grow with $n,p$, and construct confidence intervals and hypothesis tests based on the asymptotic normality result. In the case of high-dimensional linear regression,  de-biased estimators has been investigated in \cite{buhlmann2013statistical,zhang2014confidence,van2014asymptotically,javanmard2014confidence}. The convex feasibility program we proposed here can be viewed as a unified treatment for general linear inverse models. We will show that under some conditions on the sample size and the local tangent cone, 
asymptotic confidence intervals and hypothesis tests are valid for finite linear contrast $\langle v, M \rangle$ which include as a special case the individual coordinates of $M$.

\section{Local Geometric Theory: Gaussian Setting}
\label{GE.Thy}

We establish in this section a general theory of geometric inference for the linear inverse problem under the Gaussian setting where the noise vector $Z$ is Gaussian and the linear operator $\mathcal{X}$ is the Gaussian ensemble design in the following sense.
\begin{definition}[Gaussian Ensemble Design]
Let $\mathcal{X} \in \mathbb{R}^{n \times p}$ overload the matrix form of the linear operator $\mathcal{X}: \mathbb{R}^p \rightarrow \mathbb{R}^n$. $\mathcal{X}$ is Gaussian ensemble if each element is i.i.d Gaussian random variable with mean $0$ and variance $\frac{1}{n}$.
\end{definition}
Our analysis is quite different from the case by case global analysis of the Dantzig selector, Lasso and nuclear norm minimization. We show a stronger result which adapts to the local tangent cone geometry. All the analyses in our theory are non-asymptotic, and the constants are explicit. Another advantage is that the local analysis yields robustness for a given parameter (with near but not exact low complexity), as the convergence rate is captured by the geometry of the associated local tangent cone at a given $M$.
Later in Section \ref{Gen.Thy} we will show how to extend the theory to a more general setting. Without loss of generality, we assume in our analysis that the atom set $\mathcal{A}$ is scaled so that $\sup_{v \in \mathcal{A}} \|v \|_{\ell_2} = 1$. That is, the atom set $\mathcal{A}$ is embedded into the unit Euclidean ball.

\subsection{Local Geometric Upper Bound}
\label{GE.Upp.Bd}

For the upper bound analysis, we need to choose a suitable localization radius $\lambda$ (in the convex program \eqref{CST.Min}) to guarantee that the true parameter $M$ is in the feasible set with high probability. The {\bf tuning parameter}, under the Gaussian noise assumption, is chosen as
\begin{align}
\label{Tun.Param}
\lambda_{\mathcal{A}}(\mathcal{X},\sigma,n) = \frac{\sigma}{\sqrt{n}} \left\{w(\mathcal{XA})+ \delta \cdot \sup_{v\in \mathcal{A}} \|\mathcal{X} v \|_{\ell_2}  \right\} \asymp \frac{\sigma}{\sqrt{n}} w(\mathcal{XA})
\end{align}
where  $\mathcal{XA}$ is the image of the atom set under the linear operator $\mathcal{X}$, and $\delta>0$ can be chosen arbitrarily according to the probability of success we would like to attain ($\delta$ is commonly chosen at order $\sqrt{\log p}$).  $\lambda_{\mathcal{A}}(\mathcal{X},\sigma,n)$ is a global parameter that depends on the linear operator $\mathcal{X}$ and the atom set $\mathcal{A}$, but, importantly, not on the complexity of $M$.

The following theorem geometrizes the local rate of convergence in the Gaussian case. 
\begin{theorem}[Gaussian Ensemble: Convergence Rate]
\label{GE.Thm}
Suppose we observe $(\mathcal{X}, \; Y)$ as in \eqref{GLI.Model} with the Gaussian ensemble design and $Z\sim N(0, \frac{\sigma^2}{n} I_n)$. Let $\hat M$ be the solution of \eqref{CST.Min} with $\lambda$ chosen as in \eqref{Tun.Param}.
Let $0<c<1$ be a constant. For any $\delta > 0$, if
\begin{align*}
n \geq \frac{4[w(B_2^p \cap T_\mathcal{A}(M))+\delta]^2}{c^2} \vee \frac{1}{c},
\end{align*}
then  with probability at least $1- 3\exp(-\delta^2/2)$,
\begin{align*}
\| \hat{M} - M \|_{\ell_2} \leq  \frac{2\sigma}{(1-c)^2} \cdot    \frac{ \gamma_{\mathcal{A}}(M) w(\mathcal{XA}) }{\sqrt{n}},\\
\| \hat{M} - M \|_{\mathcal{A}} \leq \frac{2 \sigma}{(1-c)^2} \cdot \frac{ \gamma^2_{\mathcal{A}}(M) w(\mathcal{XA}) }{\sqrt{n}},\\
\| \mathcal{X}(\hat{M} - M) \|_{\ell_2} \leq \frac{2\sigma}{(1-c)} \cdot    \frac{ \gamma_{\mathcal{A}}(M) w(\mathcal{XA}) }{\sqrt{n}}.
\end{align*}
\end{theorem}

Theorem \ref{GE.Thm} gives bounds for the estimation error under both the $\ell_2$ norm loss and the atomic norm loss as well as for the in sample prediction error. The upper bounds are determined by the geometric quantities $w(\mathcal{XA}), \gamma_{\mathcal{A}}(M)$ and $w(B_2^p \cap T_{\mathcal{A}}(M))$. Take for example the estimation error under the $\ell_2$ loss. 
Given any $\epsilon>0$, the smallest sample size $n$  to ensure the recovery error $\|\hat{M} - M\|_{\ell_2} \le \epsilon$ with probability at least $1-3\exp(-\delta^2/2)$ is 
\begin{align*}
n \geq \max\left\{ \frac{4\sigma^2}{(1-c)^4} \cdot  \frac{\gamma^2_{\mathcal{A}}(M)w^2(\mathcal{XA})}{\epsilon^2},~ \frac{4 w^2(B_2^p \cap T_\mathcal{A}(M))}{c^2}  \right\}.
\end{align*}
That is, the minimum sample size for guaranteed statistical accuracy is driven by two geometric terms $w(\mathcal{XA}) \gamma_{\mathcal{A}}(M)$ and $w(B_2^p \cap T_\mathcal{A}(M))$. We will see in Section \ref{Univ.App} that these two rates match in a range of specific high-dimensional estimation problems. For the other two loss functions, similar calculation applies. It should be noted that Theorem \ref{GE.Thm} provides a local analysis of the performance of the estimator for a given $M$, which is quite different from a usual global analysis over a large parameter space.

The proof of Theorem \ref{GE.Thm} (and Theorem \ref{GLC.Thm} in Section \ref{Gen.Thy}) relies on the following two key lemmas. The first one is on the choice of the tuning parameter  $\lambda$ which is based on the following lemma in the Gaussian case.
\begin{lemma}[Choice of Tuning Parameter]
\label{TP.Lemma}
Consider the linear inverse model \eqref{GLI.Model} with $Z\sim N(0, \frac{\sigma^2}{n} I_n)$. For any $\delta>0$, with probability at least $1 - \exp(-\delta^2/2)$,
\begin{align}
\| \mathcal{X}^* (Z) \|_{\mathcal{A}}^* & \leq  \frac{\sigma}{\sqrt{n}} \left\{ w(\mathcal{X A})+ \delta \cdot \sup_{v\in \mathcal{A}} \|\mathcal{X} v \|_{\ell_2}  \right\}.
\end{align}
\end{lemma}

This lemma is proved in Section \ref{Sec.Pf}. The particular value of $\lambda_{\mathcal{A}}(\mathcal{X},\sigma,n)$ for a range of examples will be calculated in Section \ref{Univ.App}.

The next lemma addresses the local behavior of the linear operator $\mathcal{X}$ around the true parameter $M$ under the Gaussian ensemble design. We call a linear operator \emph{locally near-isometric} if the local isometry constants are uniformly bounded. The following lemma tells us that in the most widely used Gaussian ensemble case, the local isometry constants are guaranteed to be bounded, given the sample size $n$ is at least of order $[w(B_2^p \cap T_\mathcal{A}(M))]^2$. Hence, the difficulty of the problem is captured by the Gaussian width.
\begin{lemma}[Local Isometry Bound for Gaussian Ensemble]
\label{LIC.Lemma}
Assume the linear operator $\mathcal{X}$ is the Gaussian ensemble design. Let $0<c<1$ be a constant. For any $\delta > 0$, if
\begin{align*}
n \geq \frac{4[w(B_2^p \cap T_\mathcal{A}(M))+\delta]^2}{c^2} \vee \frac{1}{c},
\end{align*}
then with probability at least $1 - 2\exp(-\delta^2/2)$, the local isometry constants are around 1 with
\begin{align*}
\phi_{\mathcal{A}}(M,\mathcal{X}) \geq 1-c \quad\mbox{and}\quad \psi_{\mathcal{A}}(M,\mathcal{X}) \leq 1+c.
\end{align*}
\end{lemma}

\subsection{Local Geometric Inference: Confidence Intervals and Hypothesis Testing}
\label{Geo.Infer.Sec}

For statistical inference on the general linear inverse model, we would like to choose the smallest $\eta$ in \eqref{FES.Con} to ensure that, under the Gaussian ensemble design, the feasibility set for \eqref{FES.Con} is non-empty with high probability. The following theorem establishes geometric inference for Model~\eqref{GLI.Model} .

\begin{theorem}[Geometric Inference]
	\label{CI.Thm}
	Suppose we observe $(\mathcal{X}, \; Y)$ as in \eqref{GLI.Model} with the Gaussian ensemble design and $Z\sim N(0, \frac{\sigma^2}{n} I_n)$. Let $\hat{M} \in \mathbb{R}^p,\Omega \in \mathbb{R}^{p \times p}$ be the solution of \eqref{CST.Min} and \eqref{FES.Con} , and let $\tilde{M} \in \mathbb{R}^p$ be the de-biased estimator as in \eqref{DB.Est}. Assume 
	$p \geq n \succsim w^2(B_2^p \cap T_\mathcal{A}(M)).$
	If the tuning parameters $\lambda, \eta$ are chosen with
	\begin{align*}
		\lambda \asymp \frac{\sigma}{\sqrt{n}} w(\mathcal{XA}), \quad  \eta \asymp \frac{1}{\sqrt{n}} w(\mathcal{XA}), 
	\end{align*}
	convex programs \eqref{CST.Min} and \eqref{FES.Con} have non-empty feasibility set for $\Omega$ with high probability.	
	
The following decomposition
	\begin{align}
		\tilde{M} - M  = \Delta + \frac{\sigma}{\sqrt{n}} \Omega \mathcal{X}^* W
	\end{align}
	holds, where $W \sim N(0, I_n)$ is the standard Gaussian vector with
	$$
	\Omega \mathcal{X}^* W \sim N(0, \Omega \mathcal{X}^* \mathcal{X} \Omega^*).
	$$
	and $\Delta \in \mathbb{R}^p$ satisfies 
	$$ 
	\| \Delta \|_{\infty} \precsim \gamma^2_{\mathcal{A}}(M) \cdot \lambda \eta \asymp \sigma \frac{\gamma^2_{\mathcal{A}}(M) w^2(\mathcal{XA})}{n}.
	$$
	Suppose
	\begin{align*}
		\lim_{n,p \rightarrow \infty } \frac{\gamma^2_{\mathcal{A}}(M) w^2(\mathcal{XA})}{\sqrt{n}}  = 0,
	\end{align*}
	then for any $v \in \mathbb{R}^p, \|v\|_{\ell_2}=1, \|v\|_{\ell_0} \leq k$ with $k$ finite, we have the asymptotic normality for the functional $\langle v, \tilde{M}\rangle$, 
	\begin{align}
	\label{asymp.normality}
	\frac{\sqrt{n} \left(\langle v, \tilde{M}\rangle - \langle v, M\rangle\right)}{ \sigma \sqrt{ v^* [\Omega \mathcal{X}^* \mathcal{X} \Omega^*]v} } \stackrel{n,p \rightarrow \infty}{\sim} N(0,1)
	\end{align}

\end{theorem}


		It follows from Theorem \ref{CI.Thm} that a valid asymptotic $(1-\alpha)$-level confidence intervals for $M_i, 1\leq i \leq p$ (when $v$ is taken as $e_i$ in Theorem~\ref{CI.Thm}) is
		\begin{align}
		\label{CI}
		\left[\tilde{M}_i + \Phi^{-1}\left(\frac{\alpha}{2}\right) \sigma \sqrt{\frac{[\Omega \mathcal{X}^* \mathcal{X} \Omega^*]_{ii}}{n}}, \quad \tilde{M}_i + \Phi^{-1}\left(1- \frac{\alpha}{2}\right) \sigma \sqrt{\frac{[\Omega \mathcal{X}^* \mathcal{X} \Omega^*]_{ii}}{n}} \right].
		\end{align}

		If we are interested in a low-dimensional linear contrast $\langle v, M \rangle= v_0$, $\|v\|_{\ell_2} = 1, \| v \|_{\ell_0}=k$ with $k$ fixed, consider the hypothesis testing problem
		$$
		H_0: \sum_{i=1}^p v_i M_i  = v_0  ~~\text{ v.s. }~~ H_\alpha: \sum_{i=1}^p v_i M_i  \neq v_0.
		$$
		The test statistic is 
		$$
		\frac{\sqrt{n} \left(\langle v, \tilde{M} \rangle - v_0\right)}{\sigma \left( v^* [\Omega \mathcal{X}^* \mathcal{X} \Omega^*]v\right)^{1/2}}
		$$
		and under the null, it follows an asymptotic standard normal distribution as $n\rightarrow \infty$. 
		
		Similarly, the $p$-value is of the form
		$$
		2 - 2 \Phi^{-1}\left( \left| \frac{\sqrt{n} \left(\langle v, \tilde{M} \rangle - v_0\right)}{\sigma \left( v^* [\Omega \mathcal{X}^* \mathcal{X} \Omega^*]v\right)^{1/2}} \right| \right)
		$$
		as $n \rightarrow \infty$.
		
		Note the asymptotic normality holds for any finite linear contrast, and the asymptotic variance nearly achieves the Fisher information lower bound, as $\Omega$ is an estimate of the inverse of $\mathcal{X}^*\mathcal{X}$. For fixed dimension inference, Fisher information lower bound is asymptotically optimal.

		\begin{remark}{\rm
		Note that the condition for estimation consistency of the parameter $M$ under the $\ell_2$ norm  is
		$$
		\lim_{n,p \rightarrow \infty } \frac{\gamma_{\mathcal{A}}(M) w(\mathcal{XA})}{\sqrt{n}}  = 0.
		$$
		In contrast, valid confidence intervals require a stronger condition
		$$
		\lim_{n,p \rightarrow \infty } \frac{\gamma^2_{\mathcal{A}}(M) w^2(\mathcal{XA})}{\sqrt{n}}  = 0.
		$$
		In the case when $n > p$ and the Gaussian ensemble design, $\mathcal{X}^* \mathcal{X}$ is non-singular with high probability. With the choice of $\Omega = (\mathcal{X}^* \mathcal{X})^{-1}$ and $\eta = 0$, for any $i \in [p]$, the following equation
		$$
		\sqrt{n} (\tilde{M}_i - M_i) \sim N(0, \sigma^2 [ (\mathcal{X}^* \mathcal{X})^{-1} ]_{ii})
		$$
		holds non-asymptotically.
		}
	\end{remark}

\subsection{Minimax Lower Bound for Local Tangent Cone}
\label{GE.Low.Bd}

As seen in Section \ref{GE.Upp.Bd} and \ref{Geo.Infer.Sec}, the local tangent cone plays an important role in the upper bound analysis. In this section, we are interested in restricting the parameter space to the local tangent cone and seeing how the geometry of the cone affects the minimax lower bound. 

\begin{theorem}[Lower bound Based on Local Tangent Cone]
\label{GE.Low}
Suppose we observe $(\mathcal{X}, \; Y)$ as in \eqref{GLI.Model} with the Gaussian ensemble design and $Z\sim N(0, \frac{\sigma^2}{n} I_n)$. Let $M$ be the true parameter of interest. Let $0<c<1$ be a constant. For any $\delta > 0$, if
\begin{align*}
n \geq \frac{4[w(B_2^p \cap T_\mathcal{A}(M))+\delta]^2}{c^2} \vee \frac{1}{c}.
\end{align*}
Then with probability at least $1 - 2 \exp(-\delta^2/2)$, 
\begin{align*}
\inf_{\hat{M}} \sup_{M' \in T_{\mathcal{A}}(M)} \mathbb{E}_{\cdot |\mathcal{X}} \| \hat{M} - M' \|_{\ell_2}^2 & \geq  \frac{c_0 \sigma^2}{(1+c)^2} \cdot \left(  \frac{e(B_2^p \cap T_\mathcal{A}(M))}{\sqrt{n}}    \right)^2
\end{align*}
for some universal constant $c_0>0$. Here $\mathbb{E}_{\cdot | \mathcal{X}}$ stands for the conditional expectation given the design matrix $\mathcal{X}$, and the probability statement is with respect to the distribution of $\mathcal{X}$ under the Gaussian ensemble design.
\end{theorem}


In the Gaussian setting, when $n \gtrsim w^2(B_2^p \cap T_{\mathcal{A}}(M))$, we have the following observations. From Theorem \ref{GE.Thm}, the local upper bound is basically determined by $\gamma^2_{\mathcal{A}}(M) w^2(\mathcal{XA})$, which is of the rate $w^2(B_2^p \cap T_{\mathcal{A}}(M))$, as we will show in Section \ref{Univ.App} in many examples. The general relationship between these two quantities is given in Lemma \ref{Upp.Link} below.
\begin{lemma}
\label{Upp.Link}
For any atom set $\mathcal{A}$, we have the following relation
\begin{align*}
\gamma_{\mathcal{A}}(M)w(\mathcal{A}) \geq w(B_2^p \cap T_{\mathcal{A}}(M))
\end{align*}
where $w(\cdot)$ is the Gaussian width and $\gamma_{\mathcal{A}}(M)$ is defined in \eqref{Asphere.Ratio}.
\end{lemma}
Lemma \ref{Upp.Link} is proved in Appendix \ref{Sec.Tech.Lma}. 

From Theorem \ref{GE.Low}, the minimax lower bound for estimation over the local tangent cone is determined by the Sudakov estimate $e^2(B_2^p \cap T_{\mathcal{A}}(M))$. 
An interesting question is: How are the two terms $w(B_2^p \cap T_{\mathcal{A}}(M))$ and $e(B_2^p \cap T_{\mathcal{A}}(M))$ related to each other? It follows directly from  Lemma \ref{Sudakov.Lemma} that there exists a universal constant $c>0$ such that 
$
c \cdot e(B_2^p \cap T_{\mathcal{A}}(M)) \leq w(B_2^p \cap T_{\mathcal{A}}(M)) \leq 24 \int_0^\infty \sqrt{\log \mathcal{N}(B_2^p \cap T_{\mathcal{A}}(M),\epsilon)} d \epsilon.$
Thus we have shown that under the Gaussian setting, both in terms of the upper bound and lower bound, geometric complexity measures govern the difficulty of the estimation problem, through closely related quantities Gaussian width and Sudakov estimate.

\subsection{Universality of the Geometric Approach}
\label{Univ.App}

In this section we apply the general theory under the Gaussian setting to some of the actively studied high-dimensional problems mentioned in Section \ref{Sec.Intro} to illustrate the wide applicability of the theory. The detail proofs are deferred to Appendix \ref{Pf.Cor}.

\subsubsection{High Dimensional Linear Regression}

We begin by considering the high-dimensional linear regression model \eqref{sparse.model} under the assumption that the true parameter $M \in \mathbb{R}^p$ is sparse, say $\| M \|_{l_0} = s$. Our general theory applying to the $\ell_1$ minimization recovers the optimality results as in Dantzig selector and Lasso.
In this case, it can be shown that $\gamma_{\mathcal{A}}(M)w(\mathcal{A})$ and $w(B_2^p \cap T_{\mathcal{A}}(M))$ are of the same rate $\sqrt{s\log p}$. See Section \ref{Pf.Cor.1} for the detailed calculations.
The asphericity ratio $\gamma_{\mathcal{A}}(M) \geq \frac{1}{2\sqrt{s}}$ reflects the sparsity of $M$ through the local tangent cone and the Gaussian width $w(\mathcal{X A}) \asymp \sqrt{\log p}$. 
The following corollary, proved in Section \ref{Pf.Cor.1}, follows from the geometric analysis of the high-dimensional regression model.

\begin{corollary}
\label{Cor.1}
Consider the high-dimensional linear regression model \eqref{sparse.model}. Assume that $\mathcal{X} \in \mathbb{R}^{n \times p}$ is the Gaussian ensemble design and the parameter of interest $M \in \mathbb{R}^p$ is of sparsity $s$. Let $\hat{M}$ be the solution to the constrained $\ell_1$ minimization \eqref{CST.Min} with $\lambda = C_1  \sigma \sqrt{\frac{\log p}{n}}$. If $n \geq C_2 s \log p$, then
\begin{align*}
\| \hat{M} - M\|_{\ell_2} \leq C_3 \cdot \sigma \sqrt{\frac{s\log p}{n}},\\
\| \hat{M} - M \|_{\ell_1} \leq C_3 \cdot \sigma s \sqrt{\frac{\log p}{n}},\\
\| \mathcal{X}(\hat{M} - M) \|_{\ell_2} \leq C_3 \cdot \sigma  \sqrt{\frac{s\log p}{n}}.
\end{align*}
with high probability, where $C_i>0, 1\leq i\leq 3$ are some universal constants.
\end{corollary}
For $\ell_2$ norm consistency of the estimation for $M$, we require $\lim\limits_{n,p\rightarrow \infty} \frac{s\log p}{n} = 0$. However, for valid inferential guarantee, the de-biased Dantzig selector type estimator $\tilde{M}$ satisfies asymptotic normality under the condition $\lim\limits_{n,p\rightarrow \infty} \frac{s\log p}{\sqrt{n}} = 0$ through Theorem \ref{CI.Thm}. Under this condition, the confidence intervals given in \eqref{CI} has asymptotic coverage probability of  $(1-\alpha)$ and its expected length is at the parametric rate $1\over \sqrt{n}$. Furthermore, the confidence intervals do not depend on the specific value of $s$.  These properties are similar to the confidence intervals constructed in \cite{zhang2014confidence,van2014asymptotically,javanmard2014confidence}.  


\subsubsection{Low Rank Matrix Recovery}

We now consider the recovery of low-rank matrices under the trace regression model \eqref{trace.reg}. The geometric theory leads to the optimal recovery results as in nuclear norm minimization and penalized trace regression in existing literatures.

Assume the true parameter $M \in \mathbb{R}^{p \times q}$ is of low rank in the sense that $\text{rank}(M)  = r$. Let us examine the behavior of $\phi_{\mathcal{A}}(M,\mathcal{X})$, $\gamma_{\mathcal{A}}(M)$, and $\lambda_{\mathcal{A}}(\mathcal{X},\sigma,n)$.
Detailed calculations given in Section \ref{Pf.Cor.2} show that  in this case $\gamma_{\mathcal{A}}(M)w(\mathcal{A})$ and $w(B_2^p \cap T_{\mathcal{A}}(M))$ are of the same order $\sqrt{r(p+q)}$. The asphericity ratio $\gamma_{\mathcal{A}}(M) \geq \frac{1}{2\sqrt{2r}}$ characterizes the low rank structure and  the Gaussian width $w(\mathcal{X A}) \asymp \sqrt{p+q}$.  We have the following corollary for low rank matrix recovery.

\begin{corollary}
\label{Cor.2}
Consider the trace regression model \eqref{trace.reg}. Assume that $\mathcal{X} \in \mathbb{R}^{n \times p q}$ is the Gaussian ensemble design and the true parameter $M \in \mathbb{R}^{p \times q}$ is of rank $r$.
Let $\hat{M}$ be the solution to the constrained nuclear norm minimization  \eqref{CST.Min} with $\lambda = C_1  \sigma \sqrt{\frac{p+q}{n}}$. If $n \geq C_2  r(p+q)$, then, with high probability,
\begin{align*}
\| \hat{M} - M \|_{F} &\leq  C_3 \cdot \sigma \sqrt{\frac{r(p+q)}{n}},\\
\| \hat{M} - M\|_{*} &\leq C_3 \cdot \sigma r \sqrt{\frac{p+q}{n}},\\
\| \mathcal{X}(\hat{M} - M) \|_{\ell_2} &\leq C_3 \cdot \sigma  \sqrt{\frac{r(p+q)}{n}}.
\end{align*}
where $C_i>0, 1\leq i\leq 3$ are some universal constants.
\end{corollary}
For point estimation consistency of $M$ under the Frobenius norm loss, the asymptotic condition is
$\lim\limits_{n,p,q\rightarrow \infty} \frac{\sqrt{r(p+q)}}{\sqrt{n}} = 0$. For statistical inference, Theorem \ref{CI.Thm} requires $\lim\limits_{n,p,q\rightarrow \infty} \frac{r(p+q)}{\sqrt{n}} = 0$, which is essentially $n\gtrsim pq$ (sample size is larger than the dimension) for $r=1$. This phenomenon happens when the Gaussian width complexity of the rank-1 matrices is large, i.e.,  the atom set being too rich. We would like to remark that in practice, convex program \eqref{FES2.Con} can still be used for constructing confidence intervals and performing hypothesis testing. However, it is harder to provide sharp bound theoretically for the approximation error $\eta$ in \eqref{FES2.Con}, for any given $r, p, q$.

\subsubsection{Sign Vector Recovery}

We turn to the sign vector recovery model \eqref{sign.reg} where the parameter of interest $M \in \{ +1,-1 \}^p$ is a sign vector.  The convex hull of the atom set (sign vectors) is the $\ell_{\infty}$ norm ball and the corresponding $\ell_{\infty}$ norm minimization program is:
\begin{align}
\hat{M} =  \argmin_{M} \left\{ \| M \|_{\ell_{\infty}} : \; ~~ \| \mathcal{X}^*(Y - \mathcal{X} (M) ) \|_{\ell_1} \leq \lambda\right\}.
\end{align}
Applying the general theory to the $\ell_\infty$ norm minimization leads to the rates of convergence for the sign vector recovery. 
The calculations given in Section \ref{Pf.Cor.2} show that the asphericity ratio $\gamma_{\mathcal{A}}(M) \geq 1$ and  the Gaussian width $w(\mathcal{X A}) \asymp \sqrt{p}$.  Furthermore, $\gamma_{\mathcal{A}}(M)w(\mathcal{A})$ and $w(B_2^p \cap T_{\mathcal{A}}(M))$ are of the same order $\sqrt{p}$.
Applying the geometric theory to sign vector recovery leads to the following result.

\begin{corollary}
\label{Cor.3}
Consider the model \eqref{sign.reg} where the true parameter $M \in \{ +1,-1 \}^p$ is a sign vector. Assume that $\mathcal{X} \in \mathbb{R}^{n \times p}$ is the Gaussian ensemble design. Let $\hat{M}$ be the solution to the convex  program \eqref{CST.Min} with $\lambda = C_1 \sigma \sqrt{\frac{p}{n}}$. If $n \geq C_2 p$, then, with high probability,
\begin{align*}
\| \hat{M} - M \|_{\ell_2}, \| \hat{M} - M \|_{\ell_\infty}, \| \mathcal{X}(\hat{M} - M) \|_{\ell_2} \leq  C \cdot \sigma \sqrt{\frac{p}{n}},
\end{align*}
where $C>0$ is some universal constants.
\end{corollary}


\subsubsection{Orthogonal Matrix Recovery}

We now treat orthogonal matrix recovery using the spectral norm minimization.
Please see Example 4 in Section \ref{Basic.Geo} for details. The spectral norm minimization program is 
\begin{align}
\hat{M} = \argmin_{M} \left\{ \| M \|: \; ~~ \| \mathcal{X}^*(Y - \mathcal{X} (M) ) \|_* \leq \lambda\right\}.
\end{align}
Consider the same model as in trace regression, but the parameter of interest $M \in \mathbb{R}^{m \times m}$ is an orthogonal matrix. Calculations in Section \ref{Pf.Cor.4} show that $\gamma_{\mathcal{A}}(M)w(\mathcal{A})$ and $w(B_2^p \cap T_{\mathcal{A}}(M))$ are of the same rate $\sqrt{m^2}$. 
Applying the geometric analysis to orthogonal matrix recovery using the constrained spectral norm minimization yields the following.

\begin{corollary}
\label{Cor.4}
Consider the orthogonal matrix recovery model \eqref{trace.reg}. Assume that $\mathcal{X} \in \mathbb{R}^{n \times m^2}$ is the Gaussian ensemble matrix and the true parameter $M \in \mathbb{R}^{m \times m}$ is an orthogonal matrix. Let $\hat{M}$ be the solution to the program \eqref{CST.Min} with $\lambda = C_1 \sigma \sqrt{\frac{m^2}{n}}$. If $n \geq C_2 m^2$, then, with high probability,
\begin{align*}
\|  \hat{M} - M\|_{\ell_2}, \| \hat{M} - M \|, \| \mathcal{X}(\hat{M} - M) \|_{\ell_2} \leq  C \cdot \sigma \sqrt{\frac{m^2}{n}},
\end{align*}
where $C>0$ is some universal constants.
\end{corollary}

\subsubsection{Other examples}
\label{o.e}


Other examples that can be formalized under the framework of the linear inverse model include permutation matrix recovery \citep{jagabathula2011inferring}, sparse plus low rank matrix recovery \citep{candes2011robust} and matrix completion \citep{candes2009exact}. The convex relaxation of permutation matrix is double stochastic matrix; the atomic norm corresponding to sparse plus low rank atom set is the infimal convolution of  the $\ell_1$ norm and nuclear norm; for matrix completion, the design matrix can be viewed as a diagonal matrix with diagonal elements being independent Bernoulli random variables. See Section \ref{Sec.Dis} for  a discussion on further examples.


\section{Local Geometric Theory: General Setting}
\label{Gen.Thy}

We have developed in the last section a local geometric theory for the linear inverse model in the Gaussian setting. The Gaussian assumption on the design and noise enables us to carry out concrete and more specific calculations as seen in the examples given in Section \ref{Univ.App}, but the distributional assumption is not essential. In this section we extend this theory to the general setting. 

\subsection{General Local Upper Bound}
\label{Gen.Upp.Bd}

We shall consider a fixed design matrix $\mathcal{X}$. In the case of random design, results we will establish are conditional on the design. We condition on the event when the noise is controlled $\| \mathcal{X}^* (Z) \|_{\mathcal{A}}^* \leq \lambda_n$. We have seen in Section \ref{GE.Upp.Bd} how to choose $\lambda_n$ to make this happen with overwhelming probability in Lemma \ref{TP.Lemma} under Gaussian noise.

\begin{theorem}[Geometrizing Local Convergence]
\label{GLC.Thm}
Suppose we observe $(\mathcal{X}, \; Y)$ as in \eqref{GLI.Model}. Condition on the event that the noise vector $Z$ satisfies, for some given choice of localization radius $\lambda_n$
\begin{align*}
\| \mathcal{X}^* (Z) \|_{\mathcal{A}}^* \leq \lambda_n.
\end{align*}
Let $\hat{M}$ be the solution to the convex program \eqref{CST.Min} with $\lambda_n$ being the tuning parameter. Then the geometric quantities defined on the local tangent cone capture the local convergence rate for $\hat{M}$, 
\begin{align*}
\| \hat{M} - M \|_{\ell_2} \leq  \frac{2 \cdot \gamma_{\mathcal{A}}(M)}{\phi_{\mathcal{A}}^2(M,\mathcal{X})} \lambda_n,\\
\| \hat{M} - M \|_{\mathcal{A}} \leq \frac{2 \cdot \gamma^2_{\mathcal{A}}(M)}{\phi_{\mathcal{A}}^2(M,\mathcal{X})} \lambda_n ,\\
\| \mathcal{X}(\hat{M} - M) \|_{\ell_2} \leq \frac{2 \cdot \gamma_{\mathcal{A}}(M)}{\phi_{\mathcal{A}}(M,\mathcal{X})}\lambda_n 
\end{align*}
with the local asphericity ratio $\gamma_{\mathcal{A}}(M)$ defined in \eqref{Asphere.Ratio} and the local lower isometry constant $\phi_{\mathcal{A}}(M,\mathcal{X})$ defined in \eqref{Iso.Const}.
\end{theorem}

\begin{remark}{\rm
This theorem decomposes the estimation and prediction errors into three geometric components. The tuning parameter $\lambda_n$ can be regarded as a localization radius around the true parameter --- it quantifies the uncertainty in estimation for a given sample size. It is a global parameter which does not depend on the local geometry.  

The other two geometric terms depend on the local tangent cone geometry. For example, when $\mathcal{X}$ is the Gaussian ensemble design, then the local lower isometry constant $\phi_{\mathcal{A}}(M,\mathcal{X})$ is lower bounded by a constant under certain conditions, which we have shown in Lemma \ref{LIC.Lemma}. The bounds $1-c \leq \phi_{\mathcal{A}}(M,\mathcal{X}) \leq \psi_{\mathcal{A}}(M,\mathcal{X}) \leq 1+c$ hold for many different random design matrices $\mathcal{X}$. As we have seen, Section \ref{Univ.App} illustrates how this term behaves in several settings.

Another observation worth noting is that Theorem \ref{GLC.Thm} holds deterministically under the conditions on $\| \mathcal{X}^* (Z) \|_{\mathcal{A}}^*$ and $\phi_{\mathcal{A}}(M,\mathcal{X})$. It does not require distributional assumptions on noise, nor does it impose conditions on the design matrix. 
Theorem \ref{GE.Thm} can be viewed as a special case where the  local isometry constant $\phi_{\mathcal{A}}(M,\mathcal{X})$ and the local radius $\lambda_n$ are calculated explicitly under the Gaussian assumption. 

}
\end{remark}

\subsection{General Geometric Inference}

Geometric inference can also be extended for other fixed design and noise distributions. 
We can modify the convex feasibility program \eqref{FES.Con} into the following stronger form
\begin{align}
	\label{Gen.Inf.Min}
	(\Omega, \eta_n) =  \argmin_{\Omega, \eta} \left\{\eta: \; ~~ \| \mathcal{X}^* \mathcal{X} \Omega_{i\cdot}^* - e_i \|_{\mathcal{A}}^* \leq \eta, ~~\forall 1\leq i\leq p\right\}.
\end{align}
Then the following theorem holds (proof is analogous to Theorem \ref{CI.Thm}).
\begin{theorem}[Geometric Inference]
	Suppose we observe $(\mathcal{X}, \; Y)$ as in \eqref{GLI.Model}. Condition on the event that the noise vector $Z$ satisfies, for some given choice of localization radius $\lambda_n$, 
	$
	\| \mathcal{X}^* (Z) \|_{\mathcal{A}}^* \leq \lambda_n.
        $
	Let $\hat{M}$ be the solution to the convex program \eqref{CST.Min} with $\lambda_n$ being the tuning parameter. Denote $\Omega$ and $\eta_n$ as the optimal solution to the convex program \eqref{Gen.Inf.Min}, and $\tilde{M}$ as the de-biased estimator. The following decomposition
	\begin{align}
		\tilde{M} - M  = \Delta + \frac{\sigma}{\sqrt{n}} \Omega \mathcal{X}^* W
	\end{align}
	holds, where $W \sim N(0, I_n)$ is the standard Gaussian vector 
	$$
	\Omega \mathcal{X}^* W \sim N(0, \Omega \mathcal{X}^* \mathcal{X} \Omega^*)
	$$
	and $\Delta \in \mathbb{R}^p$ satisfies 
	$$ 
	\| \Delta \|_{\infty} \leq \frac{2 \cdot \gamma^2_{\mathcal{A}}(M)}{\phi_{\mathcal{A}}(M,\mathcal{X})} \cdot \lambda_n \eta_n.
	$$
\end{theorem}

\subsection{General Local Minimax Lower Bound}
\label{Gen.Low.Bd}

The lower bound given in the Gaussian case can also be extended to the general setting where the class of noise distributions contains the Gaussian distributions. We aim to geometrize the intrinsic difficulty of the estimation problem in a unified manner. 

We first present a general result for a convex cone $T$ in the parameter space, which illustrates how the Sudakov estimate, volume ratio and the design matrix affect the minimax lower bound.

\begin{theorem}[Minimax Lower Bound via Sudakov Estimate and Volume Ratio]
\label{CE.Low}
Let $T \in \mathbb{R}^p$ be a compact convex cone. The minimax lower bound for the linear inverse model~\eqref{GLI.Model}, if restricted to the cone $T$, is
\begin{align*}
\inf_{\hat{M}} \sup_{M \in T} \mathbb{E}_{\cdot |\mathcal{X}} \| \hat{M} - M \|_{\ell_2}^2 & \geq \frac{c_0 \sigma^2}{\psi^2} \cdot \left( \frac{e(B_2^p \cap T)}{\sqrt{n}}  \vee \frac{v(B_2^p \cap T)}{\sqrt{n}}  \right)^2 .
\end{align*}
where $\hat{M}$ is any measurable estimator, $\psi = \sup_{v \in B_2^p \cap T} \| \mathcal{X} (v) \|_{\ell_2}$ and $c_0$ is a universal constant. Here the notation  $\mathbb{E}_{\cdot | \mathcal{X}}$ means taking expectation conditioned on the design matrix $\mathcal{X}$. $e(\cdot)$ and $v(\cdot)$ denote the Sudakov estimate (see \eqref{Covering.Entropy}) and volume ratio (see \eqref{Volume.Ratio}).
\end{theorem}
%

Applying the theorem to the local tangent cone yields the following corollary.

\begin{corollary}[Lower bound Based on Local Tangent Cone]
\label{Gen.Cor}
Assume $T_{\mathcal{A}}(M)$ is the local tangent cone of interest. For for any measurable estimator $\hat{M}$ and for parameters $\tilde{M} \in T_{\mathcal{A}}(M)$, we have the following minimax lower bound
\begin{align*}
\inf_{\hat{M}} \sup_{M' \in T_{\mathcal{A}}(M)} \mathbb{E}_{\cdot | \mathcal{X}} \| \hat{M} - M' \|_{\ell_2}^2 & \geq  \frac{c_0 \sigma^2}{\psi_{\mathcal{A}}^2(M,\mathcal{X})} \cdot \left(  \frac{e(B_2^p \cap T_\mathcal{A}(M))}{\sqrt{n}}  \vee  \frac{v(B_2^p \cap T_\mathcal{A}(M))}{\sqrt{n}} \right)^2
\end{align*}
where $\psi_{\mathcal{A}}(M,\mathcal{X})$ is defined in \eqref{Iso.Const}. Here the notation  $\mathbb{E}_{\cdot | \mathcal{X}}$ means taking expectation conditioned on the design matrix $\mathcal{X}$.
\end{corollary}

Theorem \ref{CE.Low} and Corollary \ref{Gen.Cor} give minimax lower bounds in terms of the Sudakov estimate and volume ratio. In the Gaussian setting, Lemma \ref{LIC.Lemma} shows that the local upper isometry constant  satisfies $\psi_{\mathcal{A}}(M,\mathcal{X}) \leq 1+c$ with probability at least $1-2\exp(-\delta^2/2)$, as long as 
\begin{align*}
n \geq \frac{4[w(B_2^p \cap T_{\mathcal{A}}(M) )+\delta]^2}{c^2} \vee \frac{1}{c}.
\end{align*}
We remark that $\psi_{\mathcal{A}}(M,\mathcal{X})$ can be bounded 
under more general design matrix $\mathcal{X}$. However, under the Gaussian ensemble design, the minimum sample size $n$ to ensure that $\psi_{\mathcal{A}}(M,\mathcal{X})$ is upper bounded directly links with Gaussian width of the tangent cone.
  
From Sudakov minoration in Lemma \ref{Sudakov.Lemma}
  \begin{align}
  \label{RHS2}
c \cdot   e(B_2^p \cap T_\mathcal{A}(M)) \leq w(B_2^p \cap T_\mathcal{A}(M)).
\end{align}
Let us inspect how the right hand side of \eqref{RHS2} compares with the upper bound in Theorem \ref{GE.Thm}. Under mild conditions on $\mathcal{X}$, $w(\mathcal{XA})$ is of order $w(\mathcal{A})$, and thus the upper bound in Theorem~\ref{GE.Thm} is of the order on the left hand side in the following equation. By Lemma \ref{Upp.Link} given in Section \ref{Sec.Pf} we have
  \begin{align}
  \label{wxa}
 \gamma_{\mathcal{A}}(M)w(\mathcal{A})  \geq w(B_2^p \cap T_\mathcal{A}(M)).
  \end{align}
However, the right hand side is of the same order as the left hand side in most cases (see Section \ref{Univ.App}).
Therefore, if the Sudakov minoration is sharp up to a constant factor for the local tangent cone, $
c \cdot e(B_2^p \cap T_{\mathcal{A}}(M)) \leq w(B_2^p \cap T_{\mathcal{A}}(M)) \leq C \cdot e(B_2^p \cap T_{\mathcal{A}}(M))$
then the rate is sharp.


Applying Urysohn's inequality in Lemma \ref{Urysohn.Lemma} we have
$
v(B_2^p \cap T_{\mathcal{A}}(M)) \leq w(B_2^p \cap T_{\mathcal{A}}(M) ).$
Hence, if the reverse Urysohn's inequality holds for the local tangent cone
$
v(B_2^p \cap T_{\mathcal{A}}(M)) \geq c \cdot w(B_2^p \cap T_{\mathcal{A}}(M) )$ 
with some constant $c>0$, then the obtained rate is sharp. Please see \cite{giannopoulos2000convex} for more information on reverse Urysohn's inequality.


\section{Discussion}
\label{Sec.Dis}


This paper presents a unified geometric characterization of the local estimation rates of convergence as well as statistical inference  for high-dimensional linear inverse problems. Major technical tools used in our analysis are geometric functional analysis and concentration of measure for Gaussian processes.

The lower bound constructed in the current paper can be contrasted with the lower bounds in \cite{ye2010rate,candes2013well}. Both the above two papers consider specifically the minimax lower bound for high-dimensional linear regression. \cite{ye2010rate} related the high-dimensional linear regression problem to the normal means problem and provided the minimax lower bound under general $\ell_q$ norm. \cite{candes2013well} constructed packing/covering set via a probabilistic existence argument, then derived a minimax lower bound over sparse vectors. We focus on a more generic perspective -- lower bounds in Theorem \ref{CE.Low} holds in general for arbitrary star-shaped body $T$, which includes $\ell_p, 0\leq p\leq \infty$, balls and cones as special cases.

\section{Proofs}
\label{Sec.Pf}

We prove the main results in this section. The proofs are divided into several parts. For the upper bound of point estimation, we will first prove Theorem \ref{GLC.Thm} and then two lemmas, Lemma \ref{LIC.Lemma} and Lemma \ref{TP.Lemma} (these two Lemmas are included in Appendix \ref{Sec.Tech.Lma}).  Theorem \ref{GE.Thm} is then easy to prove. As for the statistical inference, Theorem \ref{CI.Thm} is proved based on Theorem \ref{GE.Thm}. For the lower bound of point estimation, we first prove Theorem \ref{CE.Low} and use it together with Lemma \ref{LIC.Lemma} to prove Theorem \ref{GE.Low}. Technical lemmas are deferred to Appendix \ref{Sec.Tech.Lma}. Corollaries are proved in Appendix \ref{Pf.Cor}. 


\noindent
\begin{proof}[Proof of Theorem \ref{GLC.Thm}]
The proof is clean and in a general fashion, following directly from the assumptions of the theorem and the definitions:
\begin{align*}
\| \mathcal{X}^* (Y - \mathcal{X} M) \|_{\mathcal{A}}^* &\leq \lambda_n \quad &\text{Assumption of the Theorem}\\
\| \mathcal{X}^* (Y - \mathcal{X} \hat{M}) \|_{\mathcal{A}}^* &\leq \lambda_n \quad &\text{Constraint in program} \\
\| \hat{M} \|_{\mathcal{A}} &\leq \| M \|_{\mathcal{A}} \quad &\text{Definition of minimizer}
\end{align*}
Thus we have
\begin{align}
\| \mathcal{X}^* \mathcal{X}(\hat{M}- M) \|_{\mathcal{A}}^* \leq 2\lambda_n \label{noisy.radius} ~~\text{and}~~ \hat{M} - M \in T_{\mathcal{A}}(M).
\end{align}
The first equation is due to triangle inequality and second one due to Tangent cone definition. Define $H = \hat{M} - M \in T_{\mathcal{A}}(M)$. According to the ``Cauchy-Schwarz'' \eqref{Cauchy.Schwarz} relation between atomic norm and its dual,
\begin{align*}
\| \mathcal{X}(H) \|_{\ell_2}^2  = \langle \mathcal{X}(H),\mathcal{X}(H) \rangle &= \langle \mathcal{X}^*\mathcal{X}(H),H \rangle \leq \| \mathcal{X}^*\mathcal{X}(H)\|_{\mathcal{A}}^* \|H\|_{\mathcal{A}} 
\end{align*}
Using the earlier result $\| \mathcal{X}^*\mathcal{X}(H)\|_{\mathcal{A}}^* \leq 2\lambda_n$, as well as the following two equations for any $H \in T_{\mathcal{A}}(M)$
\begin{align*}
 \phi_{\mathcal{A}}(M,\mathcal{X}) \| H \|_{\ell_2} &\leq \| \mathcal{X}(H) \|_{\ell_2} \quad &\text{local isometry constant} \\
\| H \|_{\mathcal{A}} &\leq \gamma_{\mathcal{A}}(M) \|H\|_{\ell_2} \quad &\text{local asphericity ratio}
\end{align*}
we get the following self-bounding relationship
\begin{align*}
\phi^2_{\mathcal{A}}(M,\mathcal{X}) \| H \|_{\ell_2}^2\leq \| \mathcal{X}(H) \|_{\ell_2}^2 \leq 2\lambda_n \|H\|_{\mathcal{A}} \leq 2\lambda_n \gamma_{\mathcal{A}}(M) \|H\|_{\ell_2}.
\end{align*}
The proof is then completed by simple algebra.
\end{proof}

\noindent
\begin{proof}[Proof of Theorem \ref{GE.Thm}]
Theorem \ref{GE.Thm} is a special case of Theorem \ref{GLC.Thm} under Gaussian setting, combining with Lemma \ref{LIC.Lemma} and Lemma \ref{TP.Lemma}. All we need to show is a good control of $\lambda_n$ and $\phi_{\mathcal{A}}(M,\mathcal{X})$ with probability at least $1  - 3\exp(-\delta^2/2)$ under Gaussian ensemble and Gaussian noise. We bound $\lambda_n$ with probability at least $1 - \exp(-\delta^2/2)$ via Lemma \ref{TP.Lemma}. For $\phi_{\mathcal{A}}(M,\mathcal{X})$, we can lower bound by $1-c$ with probability at least $1 - 2\exp(-\delta^2/2)$. Let's define good event to be when $$\lambda_n \leq \frac{\sigma}{\sqrt{n}} \left\{ \mathbb{E}_{g}\left[ \sup_{v\in \mathcal{A}} \langle {g}, \mathcal{X}v \rangle \right] + \delta \cdot \sup_{v\in \mathcal{A}} \|\mathcal{X} v \|_{\ell_2} \right\}$$ and $1- c \leq \phi_{\mathcal{A}}(M,\mathcal{X}) \leq \psi_{\mathcal{A}}(M,\mathcal{X}) \leq 1+c$ both hold. It is easy to see this good event holds with probability $1 - 3\exp(-\delta^2/2)$. 
Thus all we need to prove is $\max_{z \in \mathcal{A}} \| \mathcal{X} z\| \leq 1+c$ under the good event.

According to Lemma \ref{LIC.Lemma}, equation \eqref{74}'s calculation, $\max_{z \in \mathcal{A}} \| \mathcal{X} z\| \leq 1+c$ is satisfied under the condition
\begin{align*}
n \geq \frac{[w(B_2^p \cap \mathcal{A})+\delta]^2}{c^2}.
\end{align*}
As we know for any $M$,  the unit atomic norm ball ${\sf conv}(\mathcal{A})$ is contained in $2 B_2^p$ and $T_{\mathcal{A}}(M)$, which means $B_2^p \cap \mathcal{A} \subset 2B_2^p \cap T_{\mathcal{A}}(M)$, thus $w(B_2^p \cap \mathcal{A}) \leq  2 w(B_2^p \cap T_{\mathcal{A}}(M))$ (monotonic property of Gaussian width). So we have for any $M$, if 
\begin{align*}
n \geq \frac{4[w(B_2^p \cap T_{\mathcal{A}}(M))+\delta]^2}{c^2} \vee \frac{1}{c}.
\end{align*}
we have the following two equations with probability at least $1 - 2\exp(-\delta^2/2)$
\begin{align}
\label{88}
\max_{z \in \mathcal{A}} \| \mathcal{X} z\| \leq 1+c \nonumber \\
1- c \leq \phi_{\mathcal{A}}(M,\mathcal{X}) \leq \psi_{\mathcal{A}}(M,\mathcal{X}) \leq 1+c .
\end{align}
Now plugging \eqref{88} into the expression of Lemma \ref{TP.Lemma}, together with Lemma \ref{LIC.Lemma}, Theorem \ref{GLC.Thm} reduces to Theorem \ref{GE.Thm}.
\end{proof}

\noindent
	\begin{proof}[Proof of Theorem~\ref{CI.Thm}]
		Let's first prove that, with high probability, the convex feasibility program \eqref{FES.Con} is indeed feasible with $\Omega = I_n$. Equivalently we want to establish that, with high probability, for any $1\leq i\leq p$,
		\begin{align*}
		\| \mathcal{X}^* \mathcal{X} e_i - e_i \|_{\mathcal{A}}^* \leq \eta
		\end{align*}
		for some proper choice of $\eta$. Here $\mathcal{X} \in \mathbb{R}^{n \times p}$, and each entry $\mathcal{X}_{ij} \sim N(0, 1/n), i.i.d.$. Denote $g = \sqrt{n} \mathcal{X}_{\cdot i}$ as a scaling version of the $i$-th column of $\mathcal{X}$, $g \sim N(0, I_n)$ and $g'\sim N(0, I_n)$ being an independent copy.  Below the notation $O_p(\cdot)$ denotes the asymptotic order in probability.
		\begin{align*}
			\| \mathcal{X}^* \mathcal{X} e_i - e_i \|_{\mathcal{A}}^*  & = \sup_{v \in \mathcal{A}} \langle  \mathcal{X}^* \mathcal{X} e_i - e_i, v \rangle 
																	    = \sup_{v \in \mathcal{A}} \langle  \mathcal{X}^* g - e_i, v \rangle /\sqrt{n} \\
																	 & \leq \sup_{v \in \mathcal{A}} \langle \mathcal{X}_{(-i)}^* g, v \rangle /\sqrt{n} + \sup_{v \in \mathcal{A}} \left(\frac{1}{n} \sum_{j=1}^n g_{j}^2  - 1 \right) v_i \\
																	   & \stackrel{w.h.p}{\precsim} \frac{w(\mathcal{X}_{(-i)} \mathcal{A})}{\sqrt{n}} + O_p(\log p/\sqrt{n})  \quad \text{invoking Lemma~\ref{TP.Lemma}}\\
																	   & \leq \frac{w(\mathcal{X} \mathcal{A})}{\sqrt{n}} + \frac{\mathbb{E}_{g'} \sup_{v \in \mathcal{A}}\sum_{k=1}^n g'_k \mathcal{X}_{ki} v_i }{\sqrt{n}}  + O_p(\log p/\sqrt{n})\\
																	   & \leq \frac{w(\mathcal{X} \mathcal{A})}{\sqrt{n}} + \frac{\sqrt{\mathbb{E}_{g'}  (\sum_{k=1}^n g'_k \mathcal{X}_{ki} )^2 \sup_{v \in \mathcal{A}} v_i^2 }}{\sqrt{n}}  + O_p(\log p/\sqrt{n})  \\
																	   &  \leq  \frac{w(\mathcal{X} \mathcal{A})}{\sqrt{n}} + \sqrt{\frac{1+O_p(\log p/\sqrt{n})}{n}} + O_p(\log p/\sqrt{n})
		\end{align*}
		where $\mathcal{X}_{(-i)}$ is the linear operator setting $i$-th column to be all zeros. We applied Lemma~\ref{TP.Lemma} in establishing the above bounds.
		
		For the de-biased estimate $\tilde{M}$, we have
		\begin{align*}
			\tilde{M} & = \hat{M} + \Omega \mathcal{X}^* (Y - \mathcal{X}(\hat{M})) \\
			\tilde{M} - M & = (\Omega \mathcal{X}^* \mathcal{X} - I_p) (M - \hat{M}) + \Omega \mathcal{X}^* Z 
			 : = \Delta + \frac{\sigma}{\sqrt{n}}\Omega \mathcal{X}^* W.
		\end{align*}
		Then for any $1\leq i \leq p$, from the Cauchy-Schwartz relationship \eqref{Cauchy.Schwarz},
		\begin{align}
			|\Delta_i| = |\langle  \mathcal{X}^* \mathcal{X} \Omega_{i\cdot}^* - e_i, M - \hat{M} \rangle | \leq \|\mathcal{X}^* \mathcal{X} \Omega_{i\cdot}^* - e_i \|_{\mathcal{A}}^* \| M - \hat{M} \|_{\mathcal{A}} \leq \frac{\gamma^2_{\mathcal{A}}(M) w^2(\mathcal{XA})}{n}.
		\end{align}
		The last line invokes the consistency result in Theorem~\ref{GE.Thm},
		$
		\| \hat{M} - M \|_{\mathcal{A}} \precsim \frac{ \gamma^2_{\mathcal{A}}(M) w(\mathcal{XA}) }{\sqrt{n}}.
		$
		Thus we have
		$$
		\| \Delta \|_\infty \leq \frac{\gamma^2_{\mathcal{A}}(M) w^2(\mathcal{XA})}{n}.
		$$
	\end{proof}


\noindent
\begin{proof}[Proof of Theorem \ref{CE.Low} with Sudakov Estimate]
The key technical tool in proving Theorem \ref{CE.Low} is the following well-known Fano's information lemma. This version is from \citet{ma2013volume}, similar versions are provided in \citet{yang1999information,yu1997assouad, tsybakov2009introduction}, and the ideas are essentially the same.
\begin{lemma}[Fano's Lemma]
\label{Fano.Lemma}
Let $(\Theta, d(\cdot, \cdot))$ be a (pseudo) metric space and $\{ \mathbb{P}_\theta: \theta \in \Theta \}$ be a collection of probability measures. Let $r \geq 2$ be an integer and let $\mathcal{S} \subset T \subset \Theta$. Denote by $\mathcal{M}(\mathcal{S}, \epsilon, d)$ the $\epsilon$ packing set as well as the packing number of $T$ with respect to metric $d$, i.e.
\begin{align*}
\inf_{\theta,\theta'\in \mathcal{M}(\mathcal{S}, \epsilon, d)} d(\theta,\theta') \geq \epsilon.
\end{align*} 
Suppose $\beta : = \sup_{\theta,\theta'\in \mathcal{M}(\mathcal{S}, \epsilon, d)} D_{KL}(\mathbb{P}_\theta||\mathbb{P}_{\theta'})>0$. Then 
\begin{align*}
\inf_{\hat{\theta}}\sup_{\theta \in T} \mathbb{E}_\theta d^2(\hat{\theta},\theta) \geq \sup_{\mathcal{S} \subset T,\epsilon >0} \frac{\epsilon^2}{4} \left( 1 - \frac{\beta + \log 2}{\log \mathcal{M}(\mathcal{S}, \epsilon, d)}  \right).
\end{align*}
\end{lemma}
For the lower bound using {\bf Sudakov estimate}.
Recall the linear inverse model
\begin{align*}
Y = \mathcal{X}(M) + Z, \quad \text{where $Z \sim N(0, \frac{\sigma^2}{n} I_n)$}.
\end{align*}
Without loss of generality, we can assume $\sigma = 1$. The Kullback-Leiber divergence between standardized linear inverse models with different parameters under the Gaussian noise is
\begin{align*}
D_{KL}(M || M') = \frac{n \| \mathcal{X}(M) - \mathcal{X}( M') \|_{\ell_2}^2}{2}.
\end{align*}
Recall Sudakov Minoration in Lemma \ref{Sudakov.Lemma}, and denote the critical radius
\begin{align*}
\tilde{\epsilon}(B_2^p \cap T) : = \arg\max_{\epsilon} ~ \epsilon \sqrt{\log \mathcal{N}(B_2^p \cap T,\epsilon)}
\end{align*} 
Consider the cone intersected with $\ell_2$ ball with radius $\delta$,
$
K(\delta) := B_2^p(\delta) \cap T \in \mathbb{R}^p,
$
 where $\delta$ will be specified later. As before, define $\psi = \sup_{v \in B_2^p \cap T } \| \mathcal{X} (v) \|_{\ell_2}$
\begin{align*}
\sup_{M, M' \in K(\delta)} D_{KL}(M || M') \leq \frac{n (\| \mathcal{X}(M) \|_{\ell_2} + \| \mathcal{X}( M') \|_{\ell_2})^2}{2}  \leq 2 n \delta^2 \psi^2.
\end{align*}
The packing number is lower bounded by the covering number as (the last equality holds because we can scale both the set and covering ball by $\delta$)
\begin{align*}
\mathcal{M}(K(\delta), \epsilon) \geq \mathcal{N}(K(\delta), \epsilon)   = \mathcal{N}(K(1), \frac{\epsilon}{\delta}) 
\end{align*}
Applying the Fano's lemma, we have
\begin{align*}
\inf_{\hat{M}} \sup_{M \in T}  \mathbb{E}_{\cdot | \mathcal{X}} \| \hat{M} - M \|_{\ell_2}^2 & \geq \sup_{\delta >0,0<\epsilon<\delta} \frac{\epsilon^2}{4} \left( 1 - \frac{ 2 n \delta^2 \psi^2 +   \log 2}{\log \mathcal{N}(K(1), \frac{\epsilon}{\delta}) }  \right).
\end{align*}
Because $K(1) = B_2^p \cap T$, set  
\begin{align*}
\delta =  \frac{1}{2\psi} \cdot \sqrt{\frac{\log \mathcal{N}(B_2^p \cap T, \tilde{\epsilon}(B_2^p \cap T))}{n}},\quad \epsilon = \delta \cdot  \tilde{\epsilon}(B_2^p \cap T)
\end{align*}
Then we have
\begin{align*}
\inf_{\hat{M}} \sup_{M \in T}  \mathbb{E}_{\cdot | \mathcal{X}} \| \hat{M} - M \|_{\ell_2}^2 & \geq \frac{c_0}{\psi^2 n} \cdot \left( \tilde{\epsilon}(B_2^p \cap T) \sqrt{\log \mathcal{N}(B_2^p \cap T,\tilde{\epsilon}(B_2^p \cap T))}  \right)^2 .
\end{align*} 
with some universal constant $c_0$. Thus
\begin{align*}
\inf_{\hat{M}} \sup_{M \in T}  \mathbb{E}_{\cdot | \mathcal{X}} \| \hat{M} - M \|_{\ell_2}^2 & \geq \frac{c_0 \sigma^2}{\psi^2 n} \cdot \left( e(B_2^p \cap T) \right)^2 .
\end{align*} 
\end{proof}

\bibliographystyle{apalike}

\bibliography{High_Dim}

\newpage

\appendix

\section{Appendix: Supplementary Proofs for Technical Lemmas}

\label{Sec.Tech.Lma}
\noindent
\begin{proof}[Proof of Lemma \ref{TP.Lemma}]
The proof uses concentration of Lipschitz functions on Gaussian space, which is illustrated in the following lemma taken from equation (1.6) in \cite{ledoux1991probability}.

\begin{lemma}[Gaussian concentration inequality for Lipschitz functions]
Let ${g} \in \mathbb{R}^p $ be a Gaussian vector with i.i.d mean zero and variance one elements and let $F:\mathbb{R}^p \rightarrow \mathbb{R}$ be a Lipschitz function with Lipschitz constant $L$ (i.e. $F(x) - F(y)| \leq L|x - y|$ for any $x,y \in \mathbb{R}^p$, with Euclidean metric on $\mathbb{R}^p$). Then for any $\lambda>0$,
\begin{align*}
\mathbb{P}(|F({g}) - \mathbb{E}_{g} F({g}))| \geq \lambda) \leq 2\exp\left(-\frac{\lambda^2}{2L^2}\right).
\end{align*}
\end{lemma}

We would like to upper bound $\| \mathcal{X}^*(Z) \|_{\mathcal{A}}^*$ with high probability, where $Z \sim N(0,\frac{\sigma^2}{n} {\bf I}_{n})$. We have
\begin{align*}
\| \mathcal{X}^* Z \|_{\mathcal{A}}^* = \sup_{v \in \mathcal{A}} \langle \mathcal{X}^* Z, v \rangle  = \sup_{v \in \mathcal{A}}\langle Z, \mathcal{X} v \rangle.
\end{align*}
Fixing $\mathcal{X}$, we can think of $\sup_{v \in \mathcal{A}} \langle  \cdot,  \mathcal{X} v \rangle : \mathbb{R}^p \rightarrow \mathbb{R}$ as a function on the Gaussian space ${g} \sim N(0, {\bf I}_n)$ satisfying the Lipschitz constant $K_{\mathcal{X}}^{\mathcal{A}} : = \sup_{v \in \mathcal{A}} \| \mathcal{X} v\|_{\ell_2}$
\begin{align*}
|\sup_{v \in \mathcal{A}} \langle {g}_1,  \mathcal{X} v \rangle - \sup_{v \in \mathcal{A}} \langle {g}_2,  \mathcal{X} v \rangle | \leq  K_{\mathcal{X}}^{\mathcal{A}} \cdot \| {g}_1 - {g}_2\|_{\ell_2}.
\end{align*}
In fact, first fixing an $u_1 = \arg \sup_{v \in \mathcal{A}} \langle {g}_1,  \mathcal{X} v \rangle$, then
\begin{align*}
\sup_{v \in \mathcal{A}} \langle {g}_1,  \mathcal{X} v \rangle - \sup_{v \in \mathcal{A}} \langle {g}_2,  \mathcal{X} v \rangle \leq  \langle {g}_1-{g}_2,  \mathcal{X} u_1 \rangle \leq \|\mathcal{X} u_1 \|_{\ell_2} \cdot  \| {g}_1-{g}_2\|_{\ell_2}.
\end{align*}
The other side uses the same trick, fixing $u_2 = \arg \sup_{v \in \mathcal{A}} \langle {g}_2, \mathcal{X} v \rangle$ 
\begin{align*}
\sup_{v \in \mathcal{A}} \langle {g}_1,  \mathcal{X} v \rangle - \sup_{v \in \mathcal{A}} \langle {g}_2,  \mathcal{X} v \rangle \geq  \langle {g}_1-{g}_2,  \mathcal{X} u_2 \rangle \geq -\|\mathcal{X} u_2 \|_{\ell_2} \cdot  \| {g}_1-{g}_2\|_{\ell_2}.
\end{align*}
Thus we proved the Lipschitz constant is upper bounded by $K_{\mathcal{X}}^{\mathcal{A}}$. Now we can apply the concentration of Lipschitz function on Gaussian space and get
\begin{align}
\mathbb{P}(\| \mathcal{X}^* Z \|_{\mathcal{A}}^*\geq \mathbb{E}\| \mathcal{X}^* Z \|_{\mathcal{A}}^* + \lambda  ) \leq \exp\left(-\frac{n \lambda^2}{2 \sigma^2  (K_{\mathcal{X}}^{\mathcal{A}})^2}\right).
\end{align}
Thus we have with probability at least $1 - \exp(-\delta^2/2)$,
\begin{align}
\| \mathcal{X}^* Z \|& \leq  \frac{\sigma}{\sqrt{n}} \left\{ \mathbb{E}_{g}\left[ \sup_{v\in \mathcal{A}} \langle {g}, \mathcal{X}v \rangle \right] + \delta \cdot \sup_{v\in \mathcal{A}} \|\mathcal{X} v \|_{\ell_2} \right\}.
\end{align}
\end{proof}

\noindent
\begin{proof}[Proof of Lemma \ref{LIC.Lemma}]
The proof uses Gordon's method \cite{gordon1988milman}. The lower bound side part of this lemma is a modified version of the key lemma in \cite{chandrasekaran2012convex}. First let's introduce an important lemma in Gordon's analysis.
\begin{lemma}(Corollary 1.2 in \cite{gordon1988milman})
Let $\Omega$ be a closed subset of $\mathbb{S}^{p-1}$. Let $\Phi: \mathbb{R}^p \rightarrow \mathbb{R}^n$ be a random map with i.i.d. zero-mean Gaussian entries having variance one. Then
\begin{align*}
\lambda_n - w(\Omega) \leq \mathbb{E} \left[ \min_{z \in \Omega} \|\Phi z \|_{\ell_2}  \right] \leq \mathbb{E} \left[ \max_{z \in \Omega} \|\Phi z \|_{\ell_2}  \right] \leq \lambda_n + w(\Omega)
\end{align*}
where $\lambda_n = \sqrt{2}\Gamma(\frac{n+1}{2})/\Gamma(\frac{n}{2})$ satisfies $n/\sqrt{n+1} < \lambda_n < \sqrt{n}$.
\end{lemma}
Use the same step as in Lemma \ref{TP.Lemma}: for any closed subset $\Omega \in \mathbb{S}^{p-1}$, the functions $\Phi \rightarrow \min_{z \in \Omega} \|\Phi z\|_{\ell_2}$ and $\Phi \rightarrow \max_{z \in \Omega} \| \Phi z \|_{\ell_2}$ both are Lipschitz maps on Gaussian space $\Phi$ with Lipchitz constant 1:
\begin{align*}
|\min_{z \in \Omega} \|\Phi_1 z\|_{\ell_2} - \min_{z \in \Omega} \|\Phi_2 z\|_{\ell_2} | \leq \| \Phi_1 - \Phi_2 \|_F, ~~~|\max_{z \in \Omega} \|\Phi_1 z\|_{\ell_2} - \max_{z \in \Omega} \|\Phi_2 z\|_{\ell_2} | \leq \| \Phi_1 - \Phi_2 \|_F.
\end{align*}
Thus using the Lipchitz concentration in Gaussian space, we have
\begin{align*}
\mathbb{P} \left( \min_{z \in \Omega} \|\mathcal{X} z\|_{\ell_2}  \leq   \mathbb{E} [ \min_{z \in \Omega} \|\mathcal{X} z \|_{\ell_2}  ] - t \right) \leq \exp( -nt^2/2 )\\
\mathbb{P} \left( \max_{z \in \Omega} \|\mathcal{X} z\|_{\ell_2}  \geq  \mathbb{E} [ \max_{z \in \Omega} \|\mathcal{X} z \|_{\ell_2}  ] + t  \right) \leq \exp(-nt^2/2)
\end{align*}
where $\mathcal{X}$ is a Gaussian ensemble design. And we have
\begin{align*}
\mathbb{P} \left( \min_{z \in \Omega} \|\mathcal{X} z\|_{\ell_2}  \leq  1-c  \right) \leq \exp( -(\lambda_n - w(\Omega) - \sqrt{n}(1-c))^2/2 ) \\
\mathbb{P} \left( \max_{z \in \Omega} \|\mathcal{X} z\|_{\ell_2}  \geq  1+c  \right) \leq \exp(-(\sqrt{n}(1+c) - \lambda_n - w(\Omega))^2/2).
\end{align*}
Thus under the condition 
\begin{align*}
 n \geq \frac{4[w(\Omega)+\delta]^2}{c^2} \vee \frac{1}{c}
\end{align*}
we have
\begin{align}
\label{74}
\sqrt{n}(1+c) - \lambda_n - w(\Omega) \geq \sqrt{n}(1+c) - \sqrt{n} - w(\Omega) \geq 2[w(\Omega)+\delta] - w(\Omega) \geq \delta
\end{align}
and
\begin{align*}
\lambda_n - w(\Omega) - \sqrt{n}(1-c) & \geq \frac{n}{\sqrt{n+1}} -  \sqrt{n}(1-c) - w(\Omega) -\delta + \delta\\
& \geq -\frac{\sqrt{n}}{\sqrt{n+1}} \cdot \frac{1}{\sqrt{n+1}+\sqrt{n}} + \sqrt{n} c -(w(\Omega)+\delta) +\delta \\
& \geq -\frac{1}{2\sqrt{n}} + \frac{\sqrt{n}c}{2} +\delta \geq \delta.
\end{align*}
Thus
\begin{align*}
\lambda_n - w(\Omega) - \sqrt{n}(1-c) \geq \delta >0, ~~~\sqrt{n}(1+c) - \lambda_n - w(\Omega) \geq \delta >0.
\end{align*}
In fact, we proved a stronger result
\begin{align*}
\mathbb{P} \left( \min_{z \in \Omega} \|\mathcal{X} z\|_{\ell_2}  \geq  1-c  \right) \geq 1 - \exp(-\delta^2/2), ~~~\mathbb{P} \left( \max_{z \in \Omega} \|\mathcal{X} z\|_{\ell_2}  \leq  1+c  \right) \geq 1 - \exp(-\delta^2/2). 
\end{align*}
Now apply our lemma to local tangent cone $T_\mathcal{A}(M)$, observe $w(B_2^p \cap T_\mathcal{A}(M)) = w(S^{p-1} \cap T_\mathcal{A}(M))$. Now the lemma holds by plugging in the tangent cone.
\end{proof}


\noindent
\begin{proof}[Proof of Lemma \ref{Upp.Link}]
The proof requires an observation
\begin{align*}
w(\mathcal{A}) = \mathbb{E}_{g} \sup_{v\in \mathcal{A}} \langle {g},v \rangle  =  \mathbb{E}_{g} \| {g}\|_{\mathcal{A}}^* \quad \text{definition of dual norm.}
\end{align*}
Then
\begin{align}
  \label{LHSRHS}
  \gamma_{\mathcal{A}}(M)w(\mathcal{A})  & = \mathbb{E}_{g} \| {g} \|_{\mathcal{A}}^* \cdot \sup_{h \in T_\mathcal{A}(M)} \frac{\| h \|_{\mathcal{A}}}{\| h \|_{\ell_2}} 
   = \mathbb{E}_{g}  \left [ \| {g}\|_{\mathcal{A}}^* \cdot \sup_{h \in T_\mathcal{A}(M)} \frac{\| h \|_{\mathcal{A}}}{\| h \|_{\ell_2}} \right ] \\
  & = \mathbb{E}_{g} \left [ \sup_{h \in T_\mathcal{A}(M)} \frac{\| {g} \|_{\mathcal{A}}^* \| h \|_{\mathcal{A}}}{\| h \|_{\ell_2}} \right ] \\
  & \geq \mathbb{E}_{g} \sup_{h \in T_\mathcal{A}(M)} \frac{\langle {g}, h \rangle}{\| h \|_{\ell_2}} = w(B_2^p \cap T_\mathcal{A}(M)).
  \end{align}
  The last step requires the Cauchy Schwartz relationship \eqref{Cauchy.Schwarz}.
\end{proof}

\noindent
\begin{proof}[Proof of Theorem \ref{CE.Low} with Volume Ratio]
For the lower bound using {\bf volume ratio}. Recall the standardized linear inverse model
\begin{align*}
Y = \mathcal{X}(M) + Z, \quad \text{where $Z \sim N(0, \frac{\sigma^2}{n} I_n)$}.
\end{align*}
Without loss of generality, we can assume $\sigma = 1$. The Kullback-Leiber divergence between standardized linear inverse models with different parameters under the Gaussian noise is
\begin{align*}
D_{KL}(M || M') = \frac{n \| \mathcal{X}(M) - \mathcal{X}( M') \|_{\ell_2}^2}{2}.
\end{align*}
Consider the intersection of a cone $T$ with $\ell_2$ ball of radius $\delta$,
$
K(\delta) := B_2^p(\delta) \cap T \subset \mathbb{R}^p,
$
 where $\delta$ will be specified later. Defining $\psi = \sup_{v \in B_2^p \cap T } \| \mathcal{X} (v) \|_{\ell_2}$,
\begin{align*}
\sup_{M, M' \in K(\delta)} D_{KL}(M || M') \leq \frac{n (\| \mathcal{X}(M) \|_{\ell_2} + \| \mathcal{X}( M') \|_{\ell_2})^2}{2}  \leq 2 n \delta^2 \psi^2.
\end{align*}
We have the packing number lower bounded by covering number as follows:
\begin{align}
\label{94}
\mathcal{M}(K(\delta), \epsilon) \geq \mathcal{N}(K(\delta), \epsilon)  \geq \frac{{\sf vol}(K(\delta))}{{\sf vol}(B_2^p(\epsilon))} = \left( \frac{\delta}{\epsilon}  \right)^p\cdot \frac{{\sf vol}(B_2^p \cap T)}{{\sf vol}(B_2^p)}.  
\end{align}
Applying Fano's inequality of Lemma~\ref{Fano.Lemma}, we have
\begin{align*}
\inf_{\hat{M}} \sup_{M \in T} \mathbb{E}_{\cdot | \mathcal{X}} \| \hat{M} - M \|_{\ell_2}^2 & \geq \sup_{\delta >0,0<\epsilon<\delta} \frac{\epsilon^2}{4} \left( 1 - \frac{ 2 n \delta^2 \psi^2 +   \log 2}{ p  \log \left[ \frac{\delta}{\epsilon} \cdot \left( \frac{{\sf vol}(B_2^p \cap T)}{{\sf vol}(B_2^p)} \right) ^{\frac{1}{p}} \right]}  \right).
\end{align*}
If for $a>0$, $0 <b <1$ we choose 
\begin{align*}
 \delta =  \frac{1}{2\psi} \cdot \sqrt{\frac{ap}{n}},\quad \epsilon = \delta \sqrt{b} \cdot \left( \frac{{\sf vol}(B_2^p \cap T)}{{\sf vol}(B_2^p)} \right) ^{\frac{1}{p}}
\end{align*}
then we have
\begin{align*}
\inf_{\hat{M}} \sup_{M \in T} \mathbb{E}_{\cdot | \mathcal{X}} \| \hat{M} - M \|_{\ell_2}^2 & \geq \frac{ p}{\psi^2 n} \cdot \left( \frac{{\sf vol}(B_2^p \cap T)}{{\sf vol}(B_2^p)} \right) ^{\frac{2}{p}}  \cdot \sup_{a>0, 0 <b <1} \frac{ab}{4} \left( 1 - \frac{ p a + 2\log 2}{p \log \frac{1}{b} }  \right) .
\end{align*} 
As shown in \citep[equation (29)]{ma2013volume}, there is a universal constant $c_0>0$ such that
\begin{align}
\inf_{\hat{M}} \sup_{M \in T}  \mathbb{E}_{\cdot | \mathcal{X}} \| \hat{M} - M \|_{\ell_2}^2 & \geq  c_0 \cdot \frac{ p}{\psi^2 n} \cdot \left( \frac{{\sf vol}(B_2^p \cap T)}{{\sf vol}(B_2^p)} \right) ^{\frac{2}{p}}.
\end{align} 
Thus
\begin{align*}
\inf_{\hat{M}} \sup_{M \in T}  \mathbb{E}_{\cdot | \mathcal{X}} \| \hat{M} - M \|_{\ell_2}^2 & \geq \frac{c_0 \sigma^2}{\psi^2 n} \cdot \left( v(B_2^p \cap T) \right)^2 .
\end{align*} 
\end{proof}

\noindent
\begin{proof}[Proof of Theorem \ref{GE.Low} and Corollary \ref{Gen.Cor}]
Theorem \ref{GE.Low} is a special case of Theorem \ref{CE.Low}, combining with Lemma \ref{LIC.Lemma}. Plug in the general convex cone $T$ by local tangent cone $T_{\mathcal{A}}(M)$, then all we need is to upper bound $\psi_{\mathcal{A}}(M,\mathcal{X}) \leq 1+c$ with high probability. This has been done in Lemma \ref{LIC.Lemma}. Corollary \ref{Gen.Cor} is a direct application of Theorem \ref{GE.Low}.
\end{proof}

\section{Proof of Corollaries}
\label{Pf.Cor}
In this section, we denote $\hat{M}$ as the solution to the program \eqref{CST.Min} and the estimation error to be $H = \hat{M} - M$. We refer explicit calculations of Gaussian width for various local tangent cone to Section 3.4 propositions 3.10-3.14 in \cite{chandrasekaran2012convex} for simplicity of our paper.

\noindent
\begin{proof}[Proof of Corollary \ref{Cor.1}]
\label{Pf.Cor.1}
Let's calculate the rate for sparse vector recovery. We will treat the geometric terms $\gamma_{\mathcal{A}}(M), \phi_{\mathcal{A}}(M,\mathcal{X}), \lambda_{\mathcal{A}}(\mathcal{X},\sigma,n)$ separately.

For $\gamma_{\mathcal{A}}(M)$:
We know that $H$ lives in the tangent cone $T_{\mathcal{A}}(M)$. Decompose $H = H_0+H_c$ according to the support of $M$, where $\|H_0\|_{l_0} = s$ and share the same support as $M$. We have 
$$
\| M\|_{\ell_1}+\|H_c\|_{\ell_1} - \| H_0\|_{\ell_1} = \| M + H_c \|_{\ell_1}-\|H_0\|_{\ell_1} \leq \|M + H_0+H_c \|_{\ell_1} \leq \| M \|_{\ell_1} 
$$
which means $\|H\|_{\ell_1}\leq 2\|H_0\|_{\ell_1}, \| H_0 \|_{\ell_2} \leq \| H \|_{\ell_2}$.
Thus we have the following relations
\begin{align*}
\|H\|_{\ell_1} \leq 2\|H_0\|_{\ell_1} \leq 2\sqrt{s} \| H_0 \|_{\ell_2} \leq 2\sqrt{s}\| H \|_{\ell_2}
\end{align*}
Therefore,
$\frac{\|H\|_{\ell_1}}{\| H \|_{\ell_2}} \leq 2\sqrt{s}$
and thus
$\gamma_{\mathcal{A}}(M) \leq 2\sqrt{s}$.

As for $\phi_{\mathcal{A}}(M,\mathcal{X})$:
By the tangent cone calculation, we can prove $\phi_{\mathcal{A}}(M,\mathcal{X}) \geq 1-c$ with high probability if 
$$
n \geq \frac{4[w(B_2^p \cap T_\mathcal{A}(M))+\delta]^2}{c^2} \vee \frac{1}{c} \asymp s \log \frac{p}{s} 
$$ 
The last bound is from Gaussian width upper bound for local tangent cone for $s-$sparse vector.

Lastly, for $\lambda_{\mathcal{A}}(\mathcal{X},\sigma,n)$: We know the operator $\mathcal{X}$ is norm preserving in the sense that $\sup_{v\in \mathcal{A}}\|\mathcal{X}v\|_{\ell_2} \leq 1+c$. and $w(\mathcal{XA})$ is the Gaussian width of $p$ discrete points on Euclidean ball, which is at most $\sqrt{2 \log p}$ due to the behavior of maximum of Gaussian variables. 
Thus we can prove $\lambda \asymp \sigma \sqrt{\frac{\log p}{n}}$ with some proper constant is enough with high probability. 
The corollary then follows fromTheorem \ref{GE.Thm}.
\end{proof}

\noindent
\begin{proof}[Proof of Corollary \ref{Cor.2}]
\label{Pf.Cor.2}
Let's calculate the rate for low rank matrix recovery. We will bound the geometric terms $\gamma_{\mathcal{A}}(M), \phi_{\mathcal{A}}(M,\mathcal{X}), \lambda_{\mathcal{A}}(\mathcal{X},\sigma,n)$ separately.

For $\gamma_{\mathcal{A}}(M)$:
Note $H$ lives in the tangent cone $T_{\mathcal{A}}(M)$. We can write $H = H_0+H_c$ according to the span of $M$ (that is, $M = U D V^T$, $H_0$ is spanned by either $U$ as the row space or $V$ as the column space, and $H_c$ is spanned by $U^\perp$ as the row space and $V^\perp$ as the column space) with the following properties
$$
\| M\|_*+\|H_c \|_* - \| H_0 \|_* = \| M+H_c \|_* - \| H_0 \|_* \leq \| M+H \|_* \leq \| M \|_* .
$$
Thus we have $\text{rank}(H_0) \leq 2r$ and $\|H\|_{*}\leq 2\|H_0\|_{*}, \| H_0 \|_{F} \leq \| H \|_{F}$.
Thus we have the following relations
\begin{align*}
\|H\|_{*} \leq 2\|H_0\|_{*} \leq 2\sqrt{2r} \| H_0 \|_{F} \leq 2\sqrt{2r}\| H \|_{F} 
\end{align*}
We then have
$\frac{\|H\|_{*}}{\| H \|_{F}} \leq 2\sqrt{2r}$
and thus
$\gamma_{\mathcal{A}}(M) \leq 2\sqrt{2r}$.

As for $\phi_{\mathcal{A}}(M,\mathcal{X})$:
By the tangent cone calculation for rank-$r$ matrix, we can prove $\phi_{\mathcal{A}}(M,\mathcal{X}) \geq 1-c$ with high probability if 
$$
n \geq \frac{4[w(B_2^p \cap T_\mathcal{A}(M))+\delta]^2}{c^2} \vee \frac{1}{c} \asymp r(p+q-r)$$ 

At last, for $\lambda_{\mathcal{A}}(\mathcal{X},\sigma,n)$: Rank one matrix manifold is a subspace with dimension $p+q-2$. Thus $w(\mathcal{XA})$ can be bounded by $\sqrt{p+q}$ because the Gaussian width of the $p+q$-dimensional subspace is $\sqrt{p+q}$ and the linear transformation cannot enlarge the dimension. The rank one matrices are of unit Frobenius norm and the $\mathcal{X}$ is norm preserving in the sense that $\sup_{v\in \mathcal{A}}\|\mathcal{X}v\|_{\ell_2} \leq 1+c$. Putting together, we can prove $\lambda \asymp \sigma \sqrt{\frac{p+q}{n}}$ with some proper constant is enough with high probability. 

The corollary follows by putting together the geometric terms and applying the Theorem \ref{GE.Thm}.
\end{proof}

\noindent
\begin{proof}[Proof of Corollary \ref{Cor.3}]
\label{Pf.Cor.3}
As usual, we will bound the geometric terms one at a time.
For $\gamma_{\mathcal{A}}(M)$],
it is clear that $\| H \|_{\ell_2}/\|H\|_{\ell_\infty} \geq 1$ and so $\gamma_{\mathcal{A}}(M) \leq  1.$
As for $\phi_{\mathcal{A}}(M,\mathcal{X})$,
by the tangent cone calculation for sign vector, we can prove $\phi_{\mathcal{A}}(M,\mathcal{X}) \geq 1-c$ with high probability if 
$$
n \geq \frac{4[w(B_2^p \cap T_\mathcal{A}(M))+\delta]^2}{c^2} \vee \frac{1}{c} \asymp p$$

Finally, for $\lambda_{\mathcal{A}}(\mathcal{X},\sigma,n)$: $\mathcal{X}$ is norm preserving in the sense that $\sup_{v\in \mathcal{A}}\|\mathcal{X}v\|_{\ell_2} \leq 1+c$ and $w(\mathcal{XA})$ is the Gaussian width of $2^p$ discrete points on Euclidean ball, which is at most $\sqrt{2 \log 2^p}$ due to the behavior of maximum of Gaussian variables.
Thus we can prove $\lambda \asymp \sigma \sqrt{\frac{p}{n}}$ with some proper constant in front of the order is enough with high probability. The corollary now follows fromTheorem \ref{GE.Thm}.
\end{proof}

\noindent
\begin{proof}[Proof of Corollary \ref{Cor.4}]
\label{Pf.Cor.4}
We bound separately the three geometric terms $\gamma_{\mathcal{A}}(M)$, $\phi_{\mathcal{A}}(M,\mathcal{X})$, and $\lambda_{\mathcal{A}}(\mathcal{X},\sigma,n)$. 
For $\gamma_{\mathcal{A}}(M)$, it is clear that $\| H \|_{F}/\|H\| \geq 1$ and thus
$\gamma_{\mathcal{A}}(M) \leq 1$.
As for $\phi_{\mathcal{A}}(M,\mathcal{X})$,
by the tangent cone calculation for orthogonal matrix, it is easy to show that $\phi_{\mathcal{A}}(M,\mathcal{X}) \geq 1-c$ with high probability if 
$$
n \geq \frac{4[w(B_2^p \cap T_\mathcal{A}(M))+\delta]^2}{c^2} \vee \frac{1}{c} \asymp \frac{m(m-1)}{2}.$$ 

At last, for $\lambda_{\mathcal{A}}(\mathcal{X},\sigma,n)$: Orthogonal matrix manifold is a subspace with dimension $\frac{m(m-1)}{2}$. Thus $w(\mathcal{XA})$ can be bounded by $\sqrt{\frac{m(m-1)}{2}}$ because the Gaussian width of the $\frac{m(m-1)}{2}$-dimensional subspace embedded into Euclidean ball is $\sqrt{\frac{m(m-1)}{2}}$ and linear transformation cannot enlarge the dimension. Note $\mathcal{X}$ is norm preserving in the sense that $\sup_{v\in \mathcal{A}}\|\mathcal{X}v\|_{\ell_2} \leq 1+c$. Thus $\lambda \asymp \sigma \sqrt{\frac{m^2}{n}}$ with some proper constant in front of the order is enough with high probability. Applying the Theorem \ref{GE.Thm}, we can prove the corollary with simple algebra.
\end{proof}

\end{document}